\documentclass{svjour3}

\usepackage{amsmath}
\usepackage{amssymb}
\usepackage{mathrsfs}
\usepackage{stmaryrd}
\usepackage{breqn}
\usepackage{graphicx}
\usepackage{grffile}
\usepackage{pgfplots}
\usepackage{pgfplotstable}
\usepackage{cite}

\numberwithin{equation}{section}

\newcommand{\vertiii}[1]{{\left\vert\kern-0.25ex\left\vert\kern-0.25ex\left\vert #1 
    \right\vert\kern-0.25ex\right\vert\kern-0.25ex\right\vert}}
\newcommand{\Bf}[0]{\mathcal{B}}
\newcommand{\Sf}[0]{\mathcal{S}_h}
\newcommand{\Lf}[0]{\mathcal{L}}

\newcommand{\clev}[0]{I_h v}

\newcommand{\Ch}[0]{\mathcal{C}_h}
\newcommand{\Eh}[0]{\mathcal{E}_h}
\newcommand{\Gh}[0]{\mathcal{G}_h}

\newcommand{\osc}[0]{\mathrm{osc}}

\newcommand{\norm}[1]{\vertiii{#1}}
\newcommand{\ju}[1]{\left\llbracket #1 \right\rrbracket}
\newcommand{\me}[1]{\left\{\hspace{-0.2cm}\left\{ #1 \right\}\hspace{-0.2cm}\right\}}

\begin{document}

\title{Error analysis of Nitsche's mortar method}
\author{Tom Gustafsson \and Rolf Stenberg \and Juha Videman}
\institute{T.~Gustafsson \and R.~Stenberg
\at Department of Mathematics and Systems Analysis, Aalto University, 00076 Aalto, Finland\\
\email{tom.gustafsson@aalto.fi},
\email{rolf.stenberg@aalto.fi}
\and
J.~Videman
\at CAMGSD/Departamento de Matem\'atica, Instituto Superior T\'ecnico, Universidade de Lisboa, 1049-001 Lisbon, Portugal\\
\email{jvideman@math.tecnico.ulisboa.pt}}

\maketitle
\begin{abstract}
    Optimal a priori and a posteriori error estimates are derived for three
    variants of Nitsche's mortar finite elements. The analysis is
    based on the equivalence of Nitsche's method and the stabilised mixed
    method.  Nitsche's method is defined so that it is robust with respect to
    large jumps in the material and mesh parameters over the interface.
    Numerical results demonstrate the robustness of the a posteriori error
    estimators.
\end{abstract}

\section{Introduction}
Nitsche's method~\cite{Nitsche} is by now a well-established and successful
method, e.g., for domain decomposition \cite{Stenberg1998,BHS,heinrich2003nitsche}, elastic contact
problems \cite{French-survey, MR3342215, MR3342214, burman2017penalty,
chouly2015symmetric}, and as a fictitious domain method
\cite{Burman-Hansbo2012, burman_hansbo_2014, MR3633544}.
However, its
mathematical analysis has not, as yet, been entirely satisfactory. In fact, for
an elliptic problem with a variational formulation in $H^1$, the existing a
priori estimates require that the solution is in $H^s$, with $s>3/2$; cf.~\cite{BHS, chouly2015symmetric}. Moreover,
the a posteriori analysis has been based on a non-rigorous saturation
assumption; cf.~\cite{BHS, chouly2017residual}.

%However, its mathematical analysis has not, as yet,
%been entirely satisfactory.
%The usual \mbox{a priori} error estimates are
%conditional on superfluous assumptions on the regularity of the exact
%solution.
%For example, in \cite{BHS} the authors assume that the solution of the
%domain decomposition problem belongs to $H^s(\Omega)$, $s>3/2$, where $\Omega$
%is the undecomposed domain. Similarly in \cite{chouly2015symmetric} the authors assume that the
%solution of the unilateral contact problem lies in $(H^s(\Omega))^d$,
%$s>3/2$, where $d\in\{2,3\}$ is the dimension of the \mbox{body $\Omega$}.
%Moreover, in \cite{Burman-Hansbo2012} the authors limit themselves to convex
%domains $\Omega$, and thus end up assuming that the solution can be found in
%$H^2(\Omega)$.
%On the other hand, the a posteriori error analysis of Nitsche's methods has
%traditionally required non-rigorous saturation assumptions; cf.~\cite{BHS, chouly2017residual}.

%These so called saturation assumptions are motivated by the asymptotic
%properties of certain error terms present in the estimates. In the case of
%domain decomposition, the assumption is that the $L^2$-norm of the normal
%derivative of the error on the boundary can be bounded with the error in the
%energy norm inside the domain.

In our paper \cite{Stenberg1995}, we made the observation that there is a close
connection between the Nitsche's method for Dirichlet conditions and a certain
stabilised mixed finite element method, and we advocated the use of the former
since it has the advantage that it directly yields a method with an optimally
conditioned, symmetric, and positive-definite stiffness matrix. The a priori
error analysis is also very straightforward but, as understood from above, not
optimal.

The purpose of the present paper is to show that this connection can be used to
improve the error analysis of the domain decomposition problem, i.e.~we will
derive optimal error estimates, both a priori and a posteriori.
We consider three similar but distinct Nitsche's mortar methods. Two of the methods have
 appeared  previously in the literature~\cite{JS2012, JuntunenEnumath,
Juntunen2015}  and the third one is a simpler master-slave formulation where
the stabilisation term is present only on the slave side of the interface. The methods
are designed so that they are robust with respect to large jumps in the material
and mesh parameters over the interface. The robustness is achieved by a proper
scaling of the stabilisation/Nitsche terms; cf.~\cite{Stenberg1998, Juntunen2015}.

%{\color{blue} Different approaches for solving domain decomposition problems have been proposed
%over the years, cf. the monograph~\cite{TW} and all the references therein.}

The plan of the paper is the following. In the next section, we present the
model transmission problem, rewrite it in a mixed saddle point variational form
and prove its stability in appropriately chosen continuous norms.  In Section 3,
we present three different stabilised mixed finite element methods and their
respective Nitsche formulations.  In Section 4, we prove the stability of the
discrete saddle point formulations and derive optimal a priori error estimates.
In Section 5, we perform the a posteriori error analysis 
and show that the residual estimators are both reliable and efficient.
In Section 6, we report the results of our numerical computations.

%This
%yields the optimal a priori estimate stated in Theorem \ref{apriori}.  In
%Section 4, based on the stability of the continuous problem, we derive
%residual-based  a posteriori estimates and show that the residual estimators
%are both reliable and efficient. In Section 5, we report the results of our
%numerical computations.

\section{The model problem}

Suppose that the polygonal or polyhedral domain $\Omega \subset \mathbb{R}^d$, $d\in\{2,3\}$, is divided into two non-overlapping parts
$\Omega_i$, $i=1,2$, and denote their common boundary by $\Gamma =
\partial \Omega_1 \cap \partial \Omega_2$. We assume that $\partial \Gamma\subset \partial \Omega$, with $\partial \Gamma$ being the boundary of the $n-1$ dimensional manifold $\Gamma.$

We consider the problem: find 
functions $u_i$ that satisfy
\begin{equation}\label{strong} 
\begin{aligned}
    -\nabla \cdot k_i \nabla u_i &= f \quad && \text{in $\Omega_i$,} \\
    u_1 - u_2 &= 0 && \text{on $\Gamma$,} \\
    k_1 \frac{\partial u_1}{\partial n_1} + k_2 \frac{\partial u_2}{\partial n_2} &= 0 && \text{on $\Gamma$}, \\
    u_i &= 0 && \text{on $\partial \Omega_i \setminus \Gamma$,}
\end{aligned}
\end{equation}
where $k_i > 0,\  i=1,2,$ are  material parameters, $f \in L^2(\Omega)$ is a load
function and $n_i$ denote the outer normal vectors to the subdomains $\Omega_i, i=1,2$. 
In what follows we often write  $n = n_1 = -n_2$. Throughout the paper we assume that $k_1\geq k_2$.

The standard variational formulation of problem \eqref{strong} reads as follows: find $u\in H^1(\Omega)$ such that
 \begin{equation}
(k \nabla u, \nabla v)_\Omega=(f,v) _\Omega \quad \forall v\in H^1(\Omega),
\end{equation}
where  $k\vert_ {\Omega_i}=k_i$ and $u\vert_ {\overline{\Omega}_i}\, =u_i$.
 On the interface $\Gamma$, the restriction of the solution $u$ lies in the Lions--Magenes space $ H^{\frac12}_{00}(\Gamma)$
  (c.f.  \cite[Theorem 11.7, page 66]{LM} or \cite[Chapter 33]{Tartar}), with its intrinsic norm defined as
  \begin{equation}
      \|v\|_{\frac12,\Gamma}^2 =   \|v\|_{0,\Gamma}^2 +\int_\Gamma \int_\Gamma \frac{\vert  v(x)-v(y)   \vert ^2}{\vert  x-y  \vert ^d}\,\mathrm{d}x\,\mathrm{d}y   + 
  \int_\Gamma \frac{  v(x)   ^2}{ \rho( x )  }\,\mathrm{d}x,
   \end{equation}
 where $\rho(x)$ is the distance from $x$ to the boundary $\partial \Gamma$.   
 
 The mixed formulation follows from imposing the continuity condition on $ \Gamma$ in a weak form by  using the normal flux as the Lagrange multiplier, viz.
$$ \lambda =  k_1 \frac{\partial u_1}{\partial n} = -k_2 \frac{\partial u_2}{\partial n}. $$
 The Lagrange multiplier belongs to the dual space 
 $Q = \big(H^{\frac12}_{00}(\Gamma)\big)^\prime$,
 equipped with the norm
\begin{equation}
    \| \xi \|_{-\frac12,\Gamma}~= \sup_{v \in H^{\frac12}_{00}(\Gamma)} \frac{\langle v, \xi \rangle}{\|v\|_{\frac12,\Gamma}},
\end{equation}
 where $\langle \cdot, \cdot \rangle : Q^\prime \times Q \rightarrow \mathbb{R}$
stands for the duality pairing.

 Let
\begin{equation}
    V_i = \{ v \in H^1(\Omega_i) : v|_{\partial \Omega_i \setminus \Gamma} = 0 \}, \quad V = V_1 \times V_2,
\end{equation}
and define the bilinear and linear forms, $\Bf : (V \times Q) \times (V \times Q) \rightarrow \mathbb{R}$ and $\Lf : V  \rightarrow \mathbb{R}$ by
\begin{align}
    \Bf(w,\xi;v,\mu) &= \sum_{i = 1}^2 (k_i \nabla w_i, \nabla v_i)_{\Omega_i} - \langle \ju{w}, \mu \rangle - \langle \ju{v}, \xi \rangle, \\
    \Lf(v) &= \sum_{i = 1}^2 (f,v_i)_{\Omega_i},
\end{align}
where $w$ and $v$ denote the pair of functions $w = (w_1, w_2) \in V_1
\times V_2$ and $v = (v_1, v_2) \in V_1 \times V_2$. Furthermore, $\ju{w}|_\Gamma = (w_1 - w_2)|_\Gamma$ 
denotes the jump in the value of $w$ over $\Gamma$.
The mixed variational formulation of \eqref{strong}
reads as follows: find $(u, \lambda) \in V \times Q$ such that
\begin{equation}
    \label{contprob}
    \Bf(u,\lambda; v, \mu) = \Lf(v) \quad \forall (v, \mu) \in V \times Q.
\end{equation}
The norm in $V\times Q$ used in the analysis is scaled by the material parameters, viz. 

\begin{equation}
    \label{energynorm}
    \norm{(w,\xi)}^2 = \sum_{i = 1}^2\Big( k_i \|\nabla w_i\|_{0,\Omega_i}^2 + 
    \frac{1}{k_i}\|\xi\|_{-\frac12,\Gamma}^2\Big).
\end{equation}

 \begin{theorem}[Continuous stability]
    \label{contstab}
   For every $(w,\xi) \in V\times Q$ there exists $(v,\mu) \in V \times Q$ such that
    \begin{equation}
        \label{eq:discstab1}
        \Bf(w,\xi; v, \mu) \gtrsim \norm{(w,\xi)}^2
    \end{equation}
    and
    \begin{equation}
        \label{eq:discstab2}
        \norm{(v,\mu)} \lesssim \norm{(w,\xi)}.
    \end{equation}
\end{theorem}
\begin{proof}
In both subdomains, we have the inf-sup condition (cf. \cite{Babuska1973}) 
\begin{equation}
\sup_{v_i\in V_i} \frac{\langle v_i, \xi\rangle}{\Vert \nabla v_i\Vert_{0,\Omega_i}}
\geq C_i \Vert   \xi   \Vert _{-\frac12, \Gamma} \quad \forall \xi\in Q , \quad i=1,2.
\end{equation}
Therefore  \begin{equation}\label{continfsup}
\sup_{v=(v_1,v_2)\in V} \frac{\langle \ju{v}, \xi\rangle}{(\sum_{i = 1}^2 k_i \|\nabla v_i\|_{0,\Omega_i}^2)^{1/2}}
\geq C   \left(\frac{1}{k_1} + \frac{1}{k_2} \right)^{1/2} \|\xi\|_{-\frac12,\Gamma} \quad \forall \xi\in Q.
\end{equation}
The stability follows now from the Babu{\v{s}}ka--Brezzi theory \cite{Babuska1973}.
\qed
\end{proof}

\begin{remark}
Given that   $k_1\geq k_2$, it holds with some constants $C_1,C_2>0$  that
 \begin{equation}
     \label{normeq}
 C_1   \norm{(w,\xi)}^2 \leq  \sum_{i = 1}^2 k_i \|\nabla w_i\|_{0,\Omega_i}^2 + %\frac{k_1 + k_2}{k_1 k_2} \|\xi\|_{-\frac12,\Gamma}^2,
 \frac{1}{k_2} \|\xi\|_{-\frac12,\Gamma}^2\leq C_2   \norm{(w,\xi)}^2 .
\end{equation}
This defines a norm that will be used in the following for defining and analysing a "master-slave" formulation.
 \end{remark}

\section{The finite element methods}

We start by defining the stabilised mixed method. 
The subdomains $\Omega_i$ are divided into sets of non-overlapping simplices $\Ch^i$, $i
= 1,2$, with $h$ referring to the mesh parameter. The edges/facets of the elements in $\Ch^i$ are divided into two meshes: $\Eh^i$ consisting of those which are located in the interior of $\Omega_i$,  and  $\Gh^i$ of those that lie on $\Gamma$. 
Furthermore, by $\Gh^\cap$ we denote the boundary mesh
obtained by intersecting the edges/facets of $\Gh^1$ and $\Gh^2$.
In particular, each $E \in \Gh^\cap$ corresponds to a pair $(E_1, E_2) \in \Gh^1
\times \Gh^2$ such that $E = E_1 \cap E_2$.
In the subdomains, we define the finite element subspaces
\begin{equation}
    V_{i,h}  = \{ v_{i,h} \in V_i : v_{i,h}|_K \in P_{p}(K)~\forall K \in \Ch^i\}, \quad V_h = V_{1,h} \times V_{2,h}, \\
\end{equation}
where $p\geq 1$. 
The finite element space for the dual variable consists of discontinuous piecewise polynomials, also of degree $p$, defined at the intersection mesh 
$\Gh^\cap$:
\begin{equation}
        Q_h = \{ \mu_h \in Q : \mu_h|_E\,\in P_{p}(E)~\forall E \in \Gh^\cap\}.
        \label{qhspace}
\end{equation}
 
We will now introduce three slightly different stabilised finite element methods and the corresponding Nitsche's formulations for problem \eqref{strong}.

\subsection{Method I}
 
We define a bilinear form $\Bf_h : (V_h \times Q_h) \times (V_h \times Q_h) \rightarrow \mathbb{R}$  through
\begin{align}
    \Bf_h(w,\xi;v,\mu) &= \Bf(w,\xi;v,\mu) - \alpha \Sf(w,\xi;v,\mu),
\end{align}
where $\alpha > 0$ is a stabilisation parameter and
\begin{equation}
\label{stabterm}
    \Sf(w,\xi;v,\mu) = \sum_{i = 1}^2 \sum_{E \in \Gh^i} \frac{h_E}{k_i} \left(\xi - k_i \frac{\partial w_i}{\partial n}, \mu - k_i \frac{\partial v_i}{\partial n} \right)_E,
\end{equation}
a stabilising term, with $h_E$ denoting the diameter of $E \in \Gh^i$.
The first stabilised finite element
method is  written as: find $(u_h, \lambda_h) \in V_h \times Q_h$ such that
\begin{equation}
    \label{discprob}
    \Bf_h(u_h,\lambda_h; v_h, \mu_h) = \Lf(v_h) \quad \forall (v_h, \mu_h) \in V_h \times Q_h.
\end{equation}
%The mesh-dependent norm that emulates the continuous norm~\eqref{energynorm}
%and in which the stability of $\Bf_h$ is obvious (see the proof of Theorem~\ref{discstab}) reads
%\begin{equation}
%    \norm{(w,\xi)}_h^2 = \sum_{i = 1}^2 k_i \|\nabla w_i\|_{0,\Omega_i}^2 + \sum_{i=1}^2 \sum_{E \in \Gh^i} \frac{h_E}{k_i} \|\xi\|_{0,E}^2.
%\end{equation}

Note that testing with  $(0,\mu_h) \in  V_h\times Q_h $ in \eqref{discprob} yields
the equation
\begin{equation}\label{discmu}
\langle  \ju{u_h}, \mu_h \rangle +\alpha   \sum_{i = 1}^2 \sum_{E \in \Gh^i} \frac{h_E}{k_i} \left(\lambda_h - k_i \frac{\partial v_{i,h}}{\partial n}, \mu_h \right)_E  = 0   \quad \forall \mu_h\in Q_h. 
\end{equation}
Hence,  denoting by $h_i : \Gamma \rightarrow \mathbb{R}$, $i=1,2$, a local mesh
size function such that
\begin{equation}
    \quad h_i|_E = h_E \quad \forall E \in \Gh^i, \quad i = 1,2,
\end{equation}
equation \eqref{discmu} can be written as
\begin{equation}\label{discmu-II}
    \left(  \ju{u_h} +\alpha   \sum_{i = 1}^2  \frac{h_i}{k_i}\left(\lambda_h - k_i \frac{\partial u_{i,h}}{\partial n}\right), \mu_h \right)_\Gamma = 0   \quad \forall \mu_h\in Q_h. 
\end{equation}
Now, since each $ E \in \Gh^\cap$ is an intersection of a pair  $(E_1, E_2) \in \Gh^1
\times \Gh^2$ and the polynomial degree is $p$ for all variables, we obtain the 
following expression for the discrete
Lagrange multiplier
\begin{equation}\label{lamb}
    \lambda_h = \me{k \frac{\partial u_h}{\partial n}} - \beta \ju{u_h},
\end{equation}
where
\begin{equation}\label{betaterm}
 \beta = \frac{\alpha^{-1} k_1 k_2}{k_2 h_1 + k_1 h_2},
\end{equation}
and
\begin{equation}\label{convexcomb}
    \me{k \frac{\partial w}{\partial n}} = \frac{k_2 h_1}{k_2 h_1 + k_1 h_2} k_1 \frac{\partial w_1}{\partial n} + \frac{k_1 h_2}{k_2 h_1 + k_1 h_2} k_2 \frac{\partial w_2}{\partial n}.\end{equation}

Substituting expression \eqref{lamb} into the discrete variational formulation leads to  the \emph{Nitsche formulation}: find $u_h \in V_h$ such that
\begin{equation}\label{nitsche}
    a_h(u_h, v_h) = \Lf(v_h) \quad \forall v_h \in V_h,
\end{equation}
where the bilinear form $a_h$ is defined through
\begin{equation}
    a_h(w,v) =\sum_{i=1}^2 (k_i \nabla w_i, \nabla v_i)_{\Omega_i} + b_h(w,v), \end{equation}
with
\begin{equation}
\begin{aligned}
    b_h(w,v) &= \sum_{E \in \Gh^\cap} \Bigg\{ (\beta \ju{w},\ju{v})_E -\left(\gamma \ju{k\frac{\partial w}{\partial n}}, \ju{k \frac{\partial v}{\partial n}} \right)_E  \\
             &\qquad\qquad -\left(\me{k \frac{\partial w}{\partial n}}, \ju{v}\right)_E - \left(\ju{w}, \me{k \frac{\partial v}{\partial n}}\right)_E  \Bigg\},
\end{aligned}
\end{equation}
and the jump term and the function $\gamma$ are given by
\begin{equation}
    \ju{k \frac{\partial w}{\partial n}} = k_1 \frac{\partial w_1}{\partial n} - k_2 \frac{\partial w_2}{\partial n}, \quad \gamma = \frac{\alpha h_1 h_2}{k_2 h_1 + k_1 h_2}.
\end{equation}
Note that \eqref{convexcomb}  is a convex combination of two fluxes as in the
method suggested in \cite{Stenberg1998}.
The formulation \eqref{nitsche} corresponds to the method introduced in
\cite{JS2012}, and to the second method proposed for problem \eqref{strong} in
\cite[pp.~468--470]{Juntunen2015}.

\subsection{Method II: Master-slave formulation}

Assume that $k_1 \gg k_2$. The norm equivalence \eqref{normeq} suggests using   only the term from the "less rigid" subdomain $\Omega_2$ for stabilisation in \eqref{stabterm}. 
Calling $\Omega_1$ the master domain and $\Omega_2$ the slave domain
% \begin{equation}
 %   \label{energynorm}
  %  \norm{(w,\xi)}^2 \approx  \sum_{i = 1}^2 k_i \|\nabla w_i\|_{0,\Omega_i}^2 + %\frac{k_1 + k_2}{k_1 k_2} \|\xi\|_{-\frac12,\Gamma}^2,
%  \frac{1}{k_2} \|\xi\|_{-\frac12,\Gamma}^2.
% \end{equation}
 and stabilising from the slave side only, yields a mixed stabilised finite element as in \eqref{discprob} except that 
\begin{equation}
    \Sf(w,\xi;v,\mu) =  \sum_{E \in \Gh^2} \frac{h_E}{k_2} \left(\xi - k_2\frac{\partial w_2}{\partial n}, \mu - k_2\frac{\partial v_2}{\partial n} \right)_E.
    \end{equation}
Note that the space for the Lagrange multiplier is still defined by \eqref{qhspace}.

The corresponding Nitsche's formulation reads as in \eqref{nitsche} with the bilinear form $b_h$  defined simply as
\begin{equation}
\begin{aligned}
    b_h(w,v)  = \sum_{E \in \Gh^\cap} \left\{\left(\frac{k_2}{\alpha h_2} \ju{w},\ju{v}\right)_E 
     -\left(k_2 \frac{\partial w}{\partial n}, \ju{v}\right)_E - \left(\ju{w}, k_2 \frac{\partial v}{\partial n}\right)_E  \right\}.
\end{aligned}
\end{equation}

\subsection{Method III: Stabilisation using a convex combination of fluxes~\cite{Stenberg1998,BHS,Juntunen2015,JuntunenEnumath}}

Let us reformulate \eqref{discprob} by considering the stabilising term 
\begin{equation}
    \alpha  \Sf(w,\xi;v,\mu) =  \left(\beta^{-1} \left(\xi - \me{k \frac{\partial w}{\partial n}} \right), \mu -\me{k \frac{\partial v}{\partial n}}  \right)_\Gamma,
\end{equation}
where $\left\{\hspace{-0.15cm}\left\{k \frac{\partial w}{\partial n}\right\}\hspace{-0.15cm}\right\}$ denotes the convex combination \eqref{convexcomb}  and $\beta$ is defined by \eqref{betaterm}.
To derive the corresponding Nitsche's method, we proceed as above and obtain an equivalent expression for the discrete Lagrange multiplier:
\begin{equation}
    \lambda_h = \me{k \frac{\partial u_h}{\partial n}} - \beta \ju{u_h}.
\end{equation}
Substituting this back to the stabilised formulation  leads to the method  \eqref{nitsche}  with $b_h$ given by
\begin{equation}
  b_h(w,v)  =      \left(\beta \ju{w} , \ju{v}  \right)_\Gamma
-   \left( \me{k \frac{\partial w}{\partial n}}, \ \ju{v} \right)_\Gamma- 
 \left(   \ju{w}, \me{k \frac{\partial v}{\partial n}} \right)_\Gamma .
\end{equation}
This exact method was discussed before in \cite{JuntunenEnumath}. A similar
method with a slightly different definition for the convex combination of fluxes
was considered in \cite{Juntunen2015}.

\begin{remark}[On the choice of the method]
    The performance of the different methods is equal by all practical measures
    when $k_1 \gg k_2$. The variational formulation of Method III has fewer terms
    and is therefore simpler to implement than Method I.
\end{remark}

\section{A priori error analysis}

In this section, we perform  a priori error analyses of the stabilised
formulations which then, by construction, carry over to the  Nitsche's
formulations. We will perform the analysis in full detail for  Method I and briefly indicate the differences in analysing the other two methods.

 In order to prove the a priori estimate (Theorem~\ref{apriori}),
we need a stability estimate for the discrete bilinear form $\Bf_h$.  The stability
estimate is proven using Lemma~\ref{lem:invest} which follows from a scaling argument:

\begin{lemma}[Discrete trace estimate]
    \label{lem:invest}
    There exists $C_I > 0$, independent of $h$, such that
    \[
        C_I \sum_{E \in \Gh^i} \frac{h_E}{k_i} \left\|k_i \frac{\partial v_{i,h}}{\partial n} \right\|_{0,E}^2 \leq k_i \|\nabla v_{i,h}\|_{0, \Omega_i}^2 \quad \forall v_{i,h} \in V_{i,h}, \quad i = 1,2.
    \]
\end{lemma}

The discrete stability of Method I will be established in the mesh-dependent norm
\begin{equation}
 \norm{(w_h,\xi_h)}_h^2= \norm{(w_h,\xi_h)}^2 + \sum_{i=1}^2 \sum_{E \in \Gh^i} \frac{h_E}{k_i}\|\xi_h\|_{0,E}^2.
\end{equation}
Note, however, that  trivially we have
\begin{equation}
    \norm{(w_h,\xi_h)}_h \geq \norm{(w_h,\xi_h)}.
\end{equation}
\begin{theorem}[Discrete stability]
    \label{discstab}
    Suppose that $0 < \alpha < C_I$. Then for every $(w_h,\xi_h) \in V_h \times Q_h$ there exists $(v_h,\mu_h) \in V_h \times Q_h$ such that
    \begin{equation}
        \label{eq:discstab1}
        \Bf_h(w_h,\xi_h; v_h, \mu_h) \gtrsim \norm{(w_h,\xi_h)}_h^2
    \end{equation}
    and
    \begin{equation}
        \label{eq:discstab2}
        \norm{(v_h,\mu_h)} _h\lesssim \norm{(w_h,\xi_h)}_h.
    \end{equation}
\end{theorem}

\begin{proof}
    Applying the discrete trace estimate leads to  stability in the mesh-dependent part of the norm
    \begin{equation}
    \label{discmeshnorm}
    \begin{aligned}        \Bf_h(w_h, \xi_h; w_h, -\xi_h)& \geq (1-\alpha C_I^{-1}) \sum_{i = 1}^2 k_i \|\nabla w_{i,h}\|_{0, \Omega_i}^2 + \alpha \sum_{i=1}^2 \sum_{E \in \Gh^i} \frac{h_E}{k_i} \|\xi_h\|_{0,E}^2
        \\
        &
        \geq C_1 \Bigg(   \sum_{i = 1}^2 k_i \|\nabla w_{i,h}\|_{0, \Omega_i}^2 +  \sum_{i=1}^2 \sum_{E \in \Gh^i} \frac{h_E}{k_i} \|\xi_h\|_{0,E}^2\Bigg).
        \end{aligned}
    \end{equation}
    Next, we recall the steps (cf. \cite{Franca-Stenberg}) for extending the result to the continuous part of the norm.
    By the continuous inf-sup condition \eqref{continfsup},  for any $\xi_h \in Q_h$ there exists $v\in V $ such that
    \begin{equation}
  \frac{\langle \ju{v}, \xi_h\rangle}{\left(\sum_{i = 1}^2 k_i \|\nabla v_i\|_{0,\Omega_i}^2\right)^{1/2}}
\geq C   \left(\frac{1}{k_1} + \frac{1}{k_2} \right)^{1/2} \|\xi_h\|_{-\frac12,\Gamma}. 
    \end{equation}
    Consequently,  there exist positive constants $C_2, \, C_3, \, C_4$, such that for the Cl\'ement interpolant $\clev\in V_h$ of $v$ it holds
    \begin{align}
    %\begin{aligned}
        &\langle \ju{\clev}, \xi_h\rangle\geq  C_2  \left(\frac{1}{k_1} + \frac{1}{k_2} \right) \|\xi_h\|_{-\frac12,\Gamma}^2-C_3 \sum_{i=1}^2 \sum_{E \in \Gh^i} \frac{h_E}{k_i} \|\xi_h\|_{0,E}^2,   \\
        &\sum_{i = 1}^2 k_i \|\nabla v_{i,h}\|_{0,\Omega_i}^2 \leq C_4  \left(\frac{1}{k_1} + \frac{1}{k_2} \right) \|\xi_h\|_{-\frac12,\Gamma}^2. \label{clementupper}
    %\end{aligned}
    \end{align}
    Using the Cauchy--Schwarz inequality, the arithmetic-geometric mean inequality, and the discrete trace estimate (Lemma~\ref{lem:invest}), we then see that
    \begin{equation}
        \label{disclowbound1}
        \begin{aligned}
            \Bf_h(w_h,\xi_h; -\clev, 0) &= - \sum_{i=1}^2 (k_i \nabla w_{i,h}, \nabla \clev _{i})+ \langle \ju{\clev}, \xi_h \rangle \\
                                      &\qquad - \sum_{i=1}^2 \sum_{E \in \Gh^i} h_E \left(\xi_h - k_i \frac{\partial w_{i,h}}{\partial n}, \frac{\partial v_{i,h}}{\partial n} \right)_E
                                      \\ 
                                      &\geq -C_5\Bigg( \sum_{i=1}^2 k_i \|\nabla w_{i,h}\|_{0,\Omega_i}^2 +   \sum_{i=1}^2 \sum_{E \in \Gh^i} \frac{h_E}{k_i}\|\xi_h\|_{0,E}^2\Bigg) \\
                                  &\qquad + C_6 \left(\frac{1}{k_1} + \frac{1}{k_2} \right) \|\xi_h\|_{-\frac12,\Gamma}^2.
        \end{aligned}
    \end{equation}
Combining  estimates \eqref{discmeshnorm} and \eqref{disclowbound1}, we finally obtain
    \begin{align*}
        \Bf_h(w_h,\xi_h; w_h -\delta \clev, -\xi_h) & 
        \geq (C_1-\delta C_5)\Bigg( \sum_{i=1}^2 k_i \|\nabla w_{i,h}\|_{0,\Omega_i}^2 +   \sum_{i=1}^2 \sum_{E \in \Gh^i} \frac{h_E}{k_i}\|\xi_h\|_{0,E}^2\Bigg) \\
                                  &\qquad + \delta C_6 \left(\frac{1}{k_1} + \frac{1}{k_2} \right) \|\xi_h\|_{-\frac12,\Gamma}^2
                                  \\& 
                                  \geq C_7 
                                  \Bigg( \sum_{i=1}^2 k_i \|\nabla w_{i,h}\|_{0,\Omega_i}^2 +   \sum_{i=1}^2 \sum_{E \in \Gh^i} \frac{h_E}{k_i}\|\xi_h\|_{0,E}^2 
                                  \\& \qquad +
                        \left(\frac{1}{k_1} + \frac{1}{k_2} \right) \|\xi_h\|_{-\frac12,\Gamma}^2 \Bigg),
    \end{align*}
    where the last bound follows from choosing $0<\delta < C_1/C_5$.
    
    In order to obtain \eqref{eq:discstab2}, we first use the triangle inequality and \eqref{clementupper} to get
    \begin{equation*}
        \norm{(w_h-\delta \clev, -\xi_h)} \leq \norm{(w_h, \xi_h)}.
    \end{equation*}
    The claim follows by adding
    \begin{equation*}
        \sum_{i=1}^2 \sum_{E \in \Gh^i} \frac{h_E}{k_i}\|\xi_h\|_{0,E}^2
    \end{equation*}
    to the both sides of the inequality.
    \qed
\end{proof}

We will need one more lemma before we can establish an optimal a priori estimate. 
Let $f_h\in V_h$ an approximation  of $f$ and define
\begin{equation}
\osc_K(f)=h_K\Vert f-f_h\Vert_{0,K}.
\end{equation}
Moreover, for each $E \in \Gh^i$, denote by $K(E) \in \Ch^i$ the element satisfying $\partial K(E) \cap E = E$. 

\begin{lemma}
    \label{lem:lowresidual}
    For an arbitrary $(v_h, \mu_h) \in V_h \times Q_h$ it holds
    \begin{equation}
        \label{eq:lowresidual}
        \begin{aligned}
            &\Bigg(\sum_{i=1}^2 \sum_{E \in \Gh^i} \frac{h_E}{k_i} \left\| \mu_h - k_i \frac{\partial v_{i,h}}{\partial n} \right\|_{0,E}^2\Bigg)^{1/2} \\
            &\qquad \lesssim \norm{(u-v_h,\lambda-\mu_h)} + \Bigg(\sum_{i=1}^2\sum_{E \in \Gh^i} \osc_{K(E)} (f)^2\Bigg)^{1/2}.
        \end{aligned}
    \end{equation}
\end{lemma}
\begin{proof}
    Let $b_E \in P_d(E) \cap H^1_0(E)$, $E \in \Gh^1$, be the edge/facet bubble function with maximum value one.
    %Denote by $K(E) \in \Ch^1$ the element satisfying $\overline{K(E)} \cap E = E$ and
    Define $\sigma_E$ as the polynomial defined on $K(E)$ through
    \[
        \sigma_E\big|_E = \frac{h_E b_E}{k_1} \Big(\mu_h - k_1\frac{\partial v_{1,h}}{\partial n}\Big) \quad \text{and} \quad \sigma_E\big|_{\partial K(E) \setminus E} = 0.
    \]
    We have by the norm equivalence in polynomial spaces
    \begin{align*}
        \frac{h_E}{k_1}\Big\|\mu_h-k_1 \frac{\partial v_{1,h}}{\partial n}\Big\|^2_{0,E} &\lesssim \frac{h_E}{k_1}\Big\|\sqrt{b_E}\Big(\mu_h-k_1\frac{\partial v_{1,h}}{\partial n}\Big)\Big\|^2_{0,E} \\
                                                                     &=\Big( \mu_h-k_1 \frac{\partial v_{1,h}}{\partial n}, \sigma_E \Big)_E.
    \end{align*}
    Let $\sigma = \sum_{E \in \Gh^1} \sigma_E$. Testing the continuous
    variational problem with $(v_1, v_2, \mu) = (\sigma, 0, 0)$ gives $(k_1
    \nabla u_1, \nabla \sigma)_{\Omega_1} - \langle \sigma, \lambda \rangle -
    (f,\sigma)_{\Omega_1} = 0$. This leads to
    \begin{align*}
        &\sum_{E \in \Gh^1} \frac{h_E}{k_1} \Big\|\mu_h-k_1 \frac{\partial v_{1,h}}{\partial n}\Big\|^2_{0,E} \\
        &\lesssim \langle \sigma, \mu_h - \lambda \rangle + (k_1 \nabla u_1, \nabla \sigma)_{\Omega_1} - (f,\sigma)_{\Omega_1} - \sum_{E \in \Gh^1} \left( k_1 \frac{\partial v_{1,h}}{\partial n}, \sigma_E \right)_E \\
        &=\langle \sigma, \mu_h - \lambda \rangle + (k_1 \nabla u_1, \nabla \sigma)_{\Omega_1} - (f,\sigma)_{\Omega_1} \\
        &\quad- \sum_{E \in \Gh^1}\left((\nabla \cdot k_1 \nabla v_{1,h},\sigma_E)_{K(E)}+(k_1 \nabla v_{1,h}, \nabla \sigma_E)_{K(E)} \right) \\
        &=\langle \sigma, \mu_h - \lambda \rangle+(k_1 \nabla(u_1-v_{1,h}),\nabla \sigma)_{\Omega_1} + \sum_{E \in \Gh^1} (-\nabla \cdot k_1 \nabla v_{1,h} -f,\sigma_E)_{K(E)}.
    \end{align*}
    By inverse estimates % preliminaries + check last steps
    \begin{equation}
        \label{eq:invgamma}
        k_1 \|\sigma\|_{1,\Omega_1}^2~\lesssim k_1 \sum_{E \in \Gh^1} h_E^{-2} \|\sigma_E\|_{0,K(E)}^2~\lesssim \sum_{E \in \Gh^1} \frac{h_E}{k_1} \Big\|\mu_h-k_1 \frac{\partial v_{1,h}}{\partial n}\Big\|_{0,E}^2.
    \end{equation}
Using the Cauchy--Schwarz and the trace inequalities, it then follows that
    \begin{align*}
        &\sum_{E \in \Gh^1} \frac{h_E}{k_1} \Big\|\mu_h-k_1 \frac{\partial v_{1,h}}{\partial n}\Big\|^2_{0,E} \\
        &\lesssim \|\lambda - \mu_h\|_{-\frac12,\Gamma} \|\sigma\|_{\frac12,\Gamma} + k_1 \|\nabla(u_1-v_{1,h})\|_{0,\Omega_1} \|\nabla \sigma\|_{0,\Omega_1} \\
        &\quad+\sum_{E \in \Gh^1} \|\nabla \cdot k_1 \nabla v_{1,h} + f\|_{0,K(E)} \|\sigma_E\|_{0,K(E)}\\
        &\lesssim \frac{1}{\sqrt{k_1}}\|\lambda - \mu_h\|_{-\frac12,\Gamma} \sqrt{k_1} \|\sigma\|_{1,\Omega_1} + \sqrt{k_1}\|\nabla(u_1-v_{1,h})\|_{0,\Omega_1} \sqrt{k_1} \|\sigma\|_{1,\Omega_1}\\
        &\quad+\Bigg(\sum_{E \in \Gh^1} \frac{h_E^2}{k_1} \|\nabla \cdot k_1 \nabla v_{1,h} + f\|_{0,K(E)}^2\Bigg)^{1/2}\Bigg(k_1 \sum_{E \in \Gh^1} h_E^{-2} \|\sigma_E\|_{0,K(E)}^2\Bigg)^{1/2}.
    \end{align*}
    In view of the standard lower bound for interior residuals \cite{Verfurth} and the discrete inequalities \eqref{eq:invgamma},
    we conclude that
    \begin{equation}
    \begin{aligned}
        &\Bigg(
        \sum_{E \in \Gh^1} \frac{h_E}{k_1} \Big\|\mu_h-k_1 \frac{\partial v_{1,h}}{\partial n}\Big\|^2_{0,E}
        \Bigg)^{1/2} \\
        &\lesssim \sqrt{k_1} \|\nabla(u_1 - v_{1,h})\|_{0,\Omega_1} + \frac{1}{\sqrt{k_1}} \|\lambda - \mu_h\|_{-\frac12,\Gamma}
                                                                                                                    \quad + \Bigg(\sum_{E \in \Gh^1} \osc_{K(E)} (f)^2\Bigg)^{1/2}.
    \end{aligned}
    \end{equation}
    The estimate in $\Omega_2$ is proven similarly.
   Adding the estimates in $\Omega_1 $ and $\Omega_2$ leads to \eqref{eq:lowresidual}.
    \qed
\end{proof}

The proof of the a priori estimate is now straightforward.

\begin{theorem}[A priori estimate]
    The exact solution $(u,\lambda) \in V \times Q$ of \eqref{contprob} and
    the discrete solution $(u_h,\lambda_h) \in V_h \times Q_h$ of \eqref{discprob} satisfy
    \begin{equation}
        \begin{aligned}
            &\norm{(u-u_h,\lambda-\lambda_h)} \\
            &\quad \lesssim \inf_{(v_h,\mu_h)\in V_h \times Q_h} \norm{(u-v_h,\lambda-\mu_h)} + \Bigg(\sum_{i=1}^2\sum_{E \in \Gh^i} \osc_{K(E)} (f)^2\Bigg)^{1/2}.
        \end{aligned}
    \end{equation}
    \label{apriori}
\end{theorem}

\begin{proof}
    The discrete stability estimate guarantees the existence of $(w_h,\xi_h) \in V_h \times Q_h$,
    with $    
        \norm{(w_h,\xi_h)}_h =1$,
 such that
    for any $(v_h,\mu_h) \in V_h \times Q_h$ it holds
    \[
        \norm{(u_h-v_h,\lambda_h-\mu_h)} \leq \norm{(u_h-v_h,\lambda_h-\mu_h)}_h \\
       \lesssim \Bf_h(u_h-v_h,\lambda_h-\mu_h; w_h,\xi_h).
    \]
    We have
    \[
         \Bf_h(u_h-v_h,\lambda_h-\mu_h; w_h,\xi_h) \\
        = \Bf(u-v_h,\lambda-\mu_h; w_h,\xi_h) + \alpha \Sf(v_h,\mu_h; w_h, \xi_h).
    \]
    The first term above is estimated using the continuity of $\Bf$ in the continuous norm
    \begin{equation*}
     \Bf(u-v_h,\lambda-\mu_h; w_h,\xi_h)   \lesssim \norm{(u-v_h,\lambda-\mu_h)}       \cdot   \norm{(w_h,\xi_h)}\lesssim  \norm{(u-v_h,\lambda-\mu_h)} . 
    \end{equation*}
    For the second term, the Cauchy--Schwarz inequality,
    Lemma~\ref{lem:lowresidual} and the discrete trace estimate yield
    \begin{equation*}
    \begin{aligned}
    \vert   \Sf(v_h,\mu_h&; w_h, \xi_h) \vert 
    \\&\lesssim 
    \left(
    \norm{(u-v_h,\lambda-\mu_h)} + \Bigg(\sum_{i=1}^2\sum_{E \in \Gh^i} \osc_{K(E)} (f)^2\Bigg)^{1/2}\right)
      \cdot   \norm{(w_h,\xi_h)}_h
      \\
      &  \lesssim    \norm{(u-v_h,\lambda-\mu_h)} + \Bigg(\sum_{i=1}^2\sum_{E \in \Gh^i} \osc_{K(E)} (f)^2\Bigg)^{1/2}. \quad\qed
      \end{aligned}
    \end{equation*}
\end{proof}

\begin{remark}[Method II] The discrete stability can be established in the norm
\[
\left( \sum_{i = 1}^2 k_i \|\nabla w_i\|_{0,\Omega_i}^2 + 
 \frac{1}{k_2} \|\xi\|_{-\frac12,\Gamma}^2 +  \sum_{E \in \Gh^2} \frac{h_E}{k_2}\|\xi_h\|_{0,E}^2.\right)^{1/2}
\]
and,  as seen from its proof,  Lemma \ref{lem:lowresidual} is valid individually for both stabilising terms. We thus obtain the a priori estimate 
  \begin{equation}
        \begin{aligned}
            &\norm{(u-u_h,\lambda-\lambda_h)} \\
            &\quad \lesssim \inf_{(v_h,\mu_h)\in V_h \times Q_h} \norm{(u-v_h,\lambda-\mu_h)} + \Bigg(\sum_{E \in \Gh^2} \osc_{K(E)} (f)^2\Bigg)^{1/2} ,
        \end{aligned}
    \end{equation}
       where   
    \begin{equation}
\norm{(w,\xi) } =  \left( \sum_{i = 1}^2 k_i \|\nabla w_i\|_{0,\Omega_i}^2 + 
 \frac{1}{k_2} \|\xi\|_{-\frac12,\Gamma}^2 \right)^{1/2} .
\label{equinorm}
\end{equation}
\end{remark}

\begin{remark}[Method III]
The analysis of the third method is similar, albeit a bit more cumbersome. The crucial observation is that we can write
\[
\xi - \me{k \frac{\partial w}{\partial n}}   = \alpha_1 \left(\xi - k_1 \frac{\partial w_1}{\partial n}  \right) + \alpha_2 \left(\xi- k_2 \frac{\partial w_2}{\partial n}  \right),
\]
where 
\[
\alpha_1 = \frac{k_2 h_1}{k_2 h_1 + k_1 h_2}, \quad \alpha_2=  \frac{k_1 h_2}{k_2 h_1 + k_1 h_2} .
\]
Given that $0 \leq \alpha_i(x) \leq 1 \ \ \forall x\in\Gamma , i=1,2,$ and $\alpha_1+\alpha_2=1$,
it can be verified using the triangle inequality that the
a priori estimate of Theorem \ref{apriori} holds also for Method III. 
\end{remark}

\section{A posteriori estimate}

Let us  first define  local residual estimators corresponding to the finite
element solution $(u_h,\lambda_h)$ through
\begin{align}
  \label{apost1}  \eta_K^2 &= \frac{h_K^2}{k_i} \| \nabla \cdot k_i \nabla u_{i,h} + f \|_{0,K}^2, \quad K \in \Ch^i, \\
  \label{apost2}  \eta_{E,\Omega}^2 &= \frac{h_E}{k_i} \left \| \left\llbracket k_i  \frac{\partial u_{i,h}}{\partial n} \right\rrbracket \right\|_{0,E}^2, \quad E \in \Eh^i, \\
 \label{apost3}   \eta_{E,\Gamma}^2 &= \frac{h_E}{k_i} \left\| \lambda_h - k_i \frac{\partial u_{i,h}}{\partial n} \right\|_{0,E}^2+\frac{k_i}{h_E} \Vert \ju{u_h}\Vert_{0,E}^2,  \quad E \in \Gh^i.
\end{align}
with $i  =1,2$.
The global error estimator is then denoted by
\begin{equation}
    \eta^2 =\sum_{i=1}^2 \Bigg( \sum_{K \in \Ch^i} \eta_K^2 + \sum_{E \in \Eh^i} \eta_{E,\Omega}^2 +  \sum_{E \in \Gh^i} \eta_{E,\Gamma}^2 \Bigg).
\end{equation}
In the following theorem we show that the error estimator $\eta$ is both efficient and reliable. 
\begin{theorem}[A posteriori estimate] It holds that
    \begin{equation}\label{reli}
        \norm{(u-u_h,\lambda-\lambda_h)} \lesssim \eta
    \end{equation}
     and
      \begin{equation}\label{effi}
           \eta \lesssim   \norm{(u-u_h,\lambda-\lambda_h)} +  \Bigg(\sum_{i=1}^2 \sum_{K \in \Ch^i} \osc_{K} (f)^2\Bigg)^{1/2}.
    \end{equation}
\end{theorem}

\begin{proof}
    The continuous stability estimate of Theorem \ref{contstab} guarantees the existence of
    a pair $(v,\mu) \in V \times Q$, with $   \norm{(v,\mu )}=1$, that satisfies
    \begin{equation}\label{th4first}
        \norm{(u-u_h,\lambda-\lambda_h)} \lesssim \Bf(u-u_h, \lambda-\lambda_h; v, \mu).
    \end{equation}
   
    Let $\clev \in V_h$ be the Cl\'ement interpolant of $v \in V$.
    The stabilised method is consistent, thus
    \[
        \Bf_h(u-u_h,\lambda-\lambda_h; \clev,0) = 0 .
    \]
    Therefore, we can write
    \[
        \Bf(u-u_h,\lambda-\lambda_h; v,\mu) = \Bf(u-u_h, \lambda-\lambda_h; v-\clev, \mu) +\alpha \Sf(u_h,\lambda_h; \clev, 0).
    \]
 After integration by parts, the first term yields
   \begin{equation} 
   \label{apostderiv}
   \begin{aligned}
        &\Bf(u-u_h,\lambda-\lambda_h; v-\clev, \mu) \\
        &= \Lf(v-\clev) - \Bf(u_h,\lambda_h; v-\clev, \mu) \\
        &= \sum_{i=1}^2\Bigg[ \sum_{K \in \Ch^i} (\nabla \cdot k_i \nabla u_{i,h} + f_i, v_i-\clev_i)_K +  \sum_{E \in \Eh^i} \left(  \ju{k_i\frac{\partial u_{i,h}}{\partial n}} , v_i-\clev_i\right)_E\\
        &\phantom{=}~ +   \sum_{E \in \Gh^i} \left(\lambda_h - k_i \frac{\partial u_{i,h}}{\partial n}, v_i - \clev_i\right)_E \Bigg]-\langle \ju{u_h},\mu\rangle.
    \end{aligned}
    \end{equation}
    On the other hand, for the Cl\'ement interpolant it holds
    \begin{equation}\label{cle}
   k_i \Vert \nabla \clev_i \Vert_{0, \Omega_i}^2+ \sum_{K \in \Ch^i} \frac{k_i}{h_K^2} \Vert v_i-\clev_i\Vert _{0,K }^2+   \sum_{E \in \Eh^i}\frac{k_i}{h_E} \Vert v_i-\clev_i  \Vert_{0,E}^2 \lesssim k_i\Vert \nabla v _i\Vert_{0, \Omega_i}^2 .
      %  &\phantom{=}~ + \sum_{i=1}^2 \sum_{E \in \Gh^i} \left(\lambda_h - k_i \frac{\partial u_{i,h}}{\partial n}, v_i - \clev_i\right)_E
    \end{equation}
    For the first three terms in \eqref{apostderiv}, we thus get
    \begin{equation}
      \begin{aligned}
        &  \sum_{K \in \Ch^i} (\nabla \cdot k_i \nabla u_{i,h} + f_i, v_i-\clev_i)_K +  \sum_{E \in \Eh^i} \left(  \ju{k_i\frac{\partial u_{i,h}}{\partial n}} , v_i-\clev_i\right)_E\\
        &\phantom{=}~ +  \sum_{E \in \Gh^i} \left(\lambda_h - k_i \frac{\partial u_{i,h}}{\partial n}, v_i - \clev_i\right)_E \lesssim \eta  \, \Vert k_i\nabla v _i\Vert_{0, \Omega_i}\lesssim \eta.
    \end{aligned}
    \end{equation}
     The last term in \eqref{apostderiv} is estimated using the  following discrete inverse inequality for the $H^{1/2}_{00}(\Gamma) $ norm  (cf. \cite{AK,cf}) 
     \[
     \Vert v_h \Vert_{1/2, \Gamma}^2 \lesssim   \sum_{E \in \Gh^i}h_E^{-1}   \Vert v_h \Vert_{0,E}^2 \quad \forall v_h\in V_h,
     \]
     viz.
     \[
     -\langle \ju{u_h},\mu\rangle \leq  \Vert \ju{u_h}\Vert_{1/2, \Gamma}   \Vert \mu\Vert_{-1/2, \Gamma}
     \lesssim \sum_{i=1}^2 \Bigg(  \sum_{E \in \Gh^i}\frac{k_i}{h_E} \Vert \ju{u_h}  \Vert_{0,E}^2\Bigg)^{1/2}
     \sqrt{k_i} 
       \Vert \mu \Vert_{-1/2, \Gamma}.
     \]
     On the other hand, from the Cauchy--Schwarz and  the discrete trace inequalities and from \eqref{cle}, it follows that
     \begin{equation}\label{th4last}
      \Sf(u_h,\lambda_h; \clev, 0)\lesssim \eta. 
     \end{equation}
     The upper bound \eqref{reli}  can now be established by joining the above estimates.
     
     The lower bound \eqref{effi} follows from Lemma \ref{lem:lowresidual} together with standard lower bounds, cf.~\cite{Verfurth}.
     \qed
    \end{proof}

     We end this section by reiterating that the purpose of the mixed stabilised
     formulation is to perform the error analysis. We advocate the use of Nitsche's
     formulation for computations and note that substituting the discrete
     Lagrange multiplier \eqref{lamb} in the error indicators, we obtain for
     $E\in\Gh^1$
     \begin{equation}
     \label{nitscheapost1}
     \frac{h_E}{k_1} \left\| \lambda_h - k_1\frac{\partial u_{1,h}}{\partial n}\right\|_{0,E}^2
     =   \frac{h_E}{k_1} \left\| \alpha_2  \ju{k \frac{\partial u_h}{\partial n}}  + \beta \ju{u_h} \right\|_{0,E}^2
     \end{equation}
     and for $E\in\Gh^2$ 
     \begin{equation}
          \label{nitscheapost2}
     \frac{h_E}{k_2} \left\| \lambda_h - k_2 \frac{\partial u_{2,h}}{\partial n}\right\|_{0,E}^2=
          \frac{h_E}{k_2} \left\| \alpha_1  \ju{k \frac{\partial u_h}{\partial n}}   -\beta \ju{u_h} \right\|_{0,E}^2.
     \end{equation}

     \begin{remark}[Method II]
The local estimators $\eta_K, \eta_{E,\Omega} $ and $\eta_{E,\Gamma}$ are defined through  \eqref{apost1}--\eqref{apost3} and the estimates \eqref{reli} and \eqref{effi} hold true  in the norm \eqref{equinorm}. After substituting the discrete Lagrange multiplier
\[
\lambda_h = k_2  \frac{\partial u_{2,h}}{\partial n}-\alpha^{-1}\frac{k_2}{h_2} \ju{u_h}
\]
 into the error indicators, we obtain for the corresponding Nitsche's formulation we obtain  we obtain for
     $E\in\Gh^1$
     \begin{equation}
     \frac{h_E}{k_1} \left\| \lambda_h - k_1\frac{\partial u_{1,h}}{\partial n}\right\|_{0,E}^2
     =   \frac{h_E}{k_1} \left\|   \ju{k \frac{\partial u_h}{\partial n}}  + \alpha^{-1} \frac{k_2}{h_2}  \ju{u_h} \right\|_{0,E}^2
     \end{equation}
     and for $E\in\Gh^2$ 
     \begin{equation}
     \frac{h_E}{k_2} \left\| \lambda_h - k_2\frac{\partial u_{2,h}}{\partial n}\right\|_{0,E}^2=
         \alpha^{-2} \frac{k_2}{h_E} \left\| \ju{u_h} \right\|_{0,E}^2.
     \end{equation}
 \end{remark}

\begin{remark}[Method III]
    Once again the local estimators for the stabilised method are defined as in  \eqref{apost1}--\eqref{apost3} and the a posteriori estimates \eqref{reli} and \eqref{effi} hold true.
    In the Nitsche's formulation, the error indicators depending on the Lagrange multiplier are given by \eqref{nitscheapost1} and \eqref{nitscheapost2}.
\end{remark}

%******
%
%For the a posteriori estimate we have to estimate as follows
%\begin{equation}
%  \vert  \Sf(u_h\lambda_h;v,\mu) \vert= \Big\vert  \left(\beta^{-1}\Big(\lambda_h- \me{k \frac{\partial u_h}{\partial n}} \Big), \mu -\me{k \frac{\partial v}{\partial n}}  \right)_\Gamma 
%  \Big\vert 
%  =
%   \Big\vert  \left(\ju{u_h}, \mu -\me{k \frac{\partial v}{\partial n}}  \right)_\Gamma 
%  \Big\vert 
%  .
%\end{equation}

\section{Numerical results}

We experiment with the proposed method by solving the domain decomposition
problem adaptively with $\Omega_1 = (0,1)^2$, $\Omega_2 = (1,2) \times (0,1)$,
$f = 1$, $\alpha = 10^{-2}$ and linear elements. After each solution we mark a triangle $K \in \Ch^i$,
$i=1,2$, for refinement if it satisfies $\mathcal{E}_K > \theta
\max_{K^\prime \in \Ch^1 \cup \Ch^2} \mathcal{E}_{K^\prime}$ where $\theta =
\tfrac{1}{\sqrt{2}}$ and
\begin{align*}
    \mathcal{E}_K^2 &= \frac{h_K^2}{k_i} \| \nabla \cdot k_i \nabla u_{i,h} + f \|_{0,K}^2 + \frac12 \sum_{E \subset \partial K \setminus \Gamma} \frac{h_E}{k_i} \left \| \left\llbracket k_i \frac{\partial u_{i,h}}{\partial n} \right\rrbracket \right\|_{0,E}^2 \qquad \\
                    %&\qquad + \frac{h_E}{k_i} \left\| \lambda_h - k_i \frac{\partial u_{i,h}}{\partial n} \right\|_{0,\partial K \cap \Gamma}^2+\frac{k_i}{h_E} \Vert \ju{u_h}\Vert_{0, \partial K \cap \Gamma}^2 \quad \forall K \in \Ch^i.
                    &\qquad + \sum_{E \subset \partial K \cap \Gamma} \left\{\frac{h_E}{k_i} \left\| \lambda_h - k_i \frac{\partial u_{i,h}}{\partial n} \right\|_{0,E}^2+\frac{k_i}{h_E} \Vert \ju{u_h}\Vert_{0,E}^2\right\} \quad \forall K \in \Ch^i.
\end{align*}
The set of marked elements is refined using the red-green-blue strategy,
see e.g.~Bartels~\cite{bartels2016numerical}. 

Changing the material parameters from $(k_1,k_2)=(1,1)$ to $(k_1,k_2)=(10,0.1)$, and
finally to $(k_1,k_2) = (0.1, 10)$ produces adaptive meshes where the domain
with a smaller material parameter receives more elements, see
Figures~\ref{fig:1}--\ref{fig:3}. This is  in accordance with results on
adaptive methods for linear elastic contact problems,
%where the mesh of the less
%rigid body is refined,
see e.g.~Wohlmuth~\cite{wohlmuth2011variationally} where it is demonstrated
that softer the material, more the respective domain is refined.

Next we solve the domain decomposition problem in an L-shaped domain with $\Omega_1 = (0,1)^2$, $\Omega_2 = (1,2) \times (0,2)$ and $k_1=k_2=1$.
The resulting sequence of meshes is depicted in Figure~\ref{fig:lshaped} and the
global error estimator as a function of the number of degrees-of-freedom $N$ is given in
Figure~\ref{fig:convergence}. Note that the exact solution is in $H^{5/3-\varepsilon}$,~$\varepsilon>0$,
in the neighbourhood of the reentrant corner which limits the convergence rate
of uniform refinements to $O(N^{-1/3})$.

We finally remark that Methods II and III yield very similar numerical results.

%\begin{figure}[h!]
%    \includegraphics[width=\textwidth]{mortar_initmesh-crop.pdf} 
%    \caption{The initial meshes.}
%    \label{fig:initmesh}
%\end{figure}
%
%\begin{figure}[h!]
%    \includegraphics[width=\textwidth]{mortar_final_1.0_1.0-crop.pdf}\\[0.3cm]
%    \includegraphics[width=\textwidth]{mortar_final_10.0_0.1-crop.pdf}\\[0.3cm]
%    \includegraphics[width=\textwidth]{mortar_final_0.1_10.0-crop.pdf} 
%    \caption{The final meshes with $k_1=k_2=1$; $k_1=10$, $k_2=0.1$; and $k_1=0.1$, $k_2=10$.}
%    \label{fig:comparison}
%\end{figure}

\begin{figure}[h!]
    \includegraphics[width=0.49\textwidth]{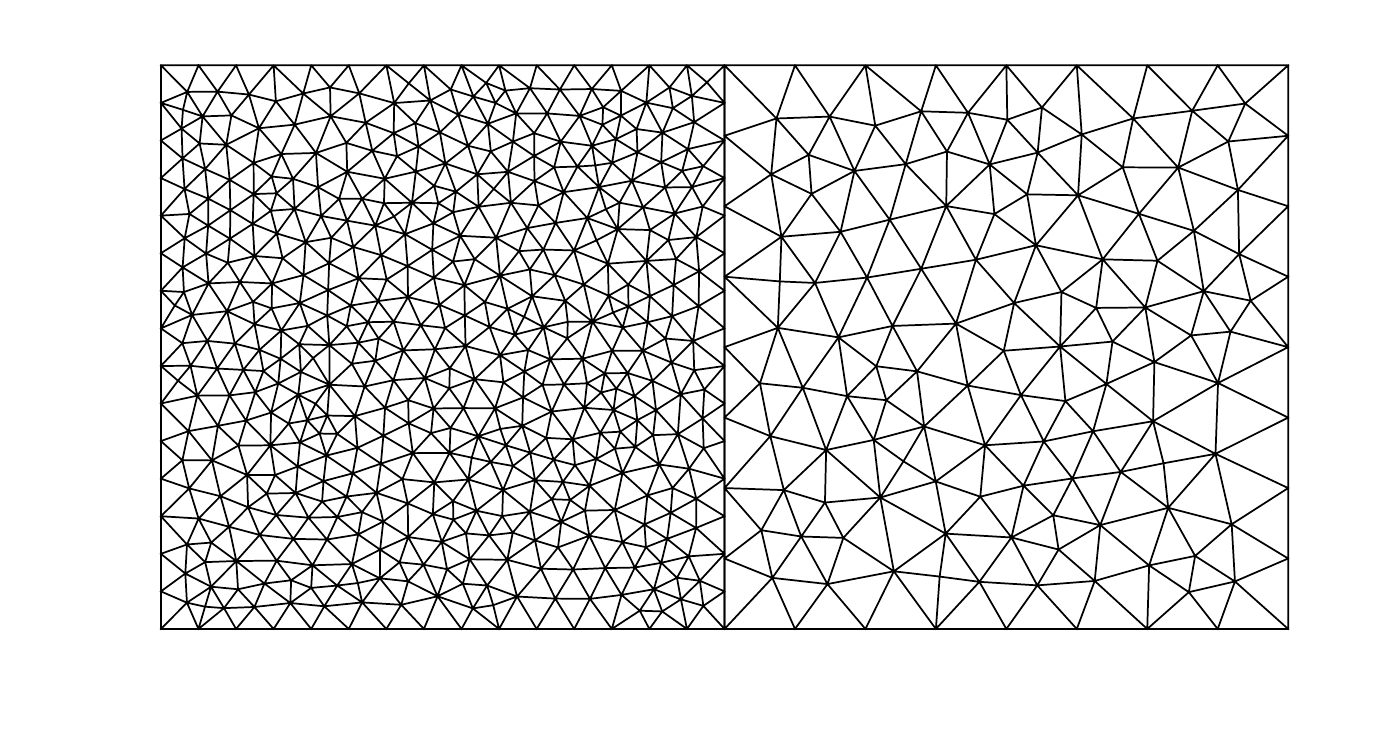}
    \includegraphics[width=0.49\textwidth]{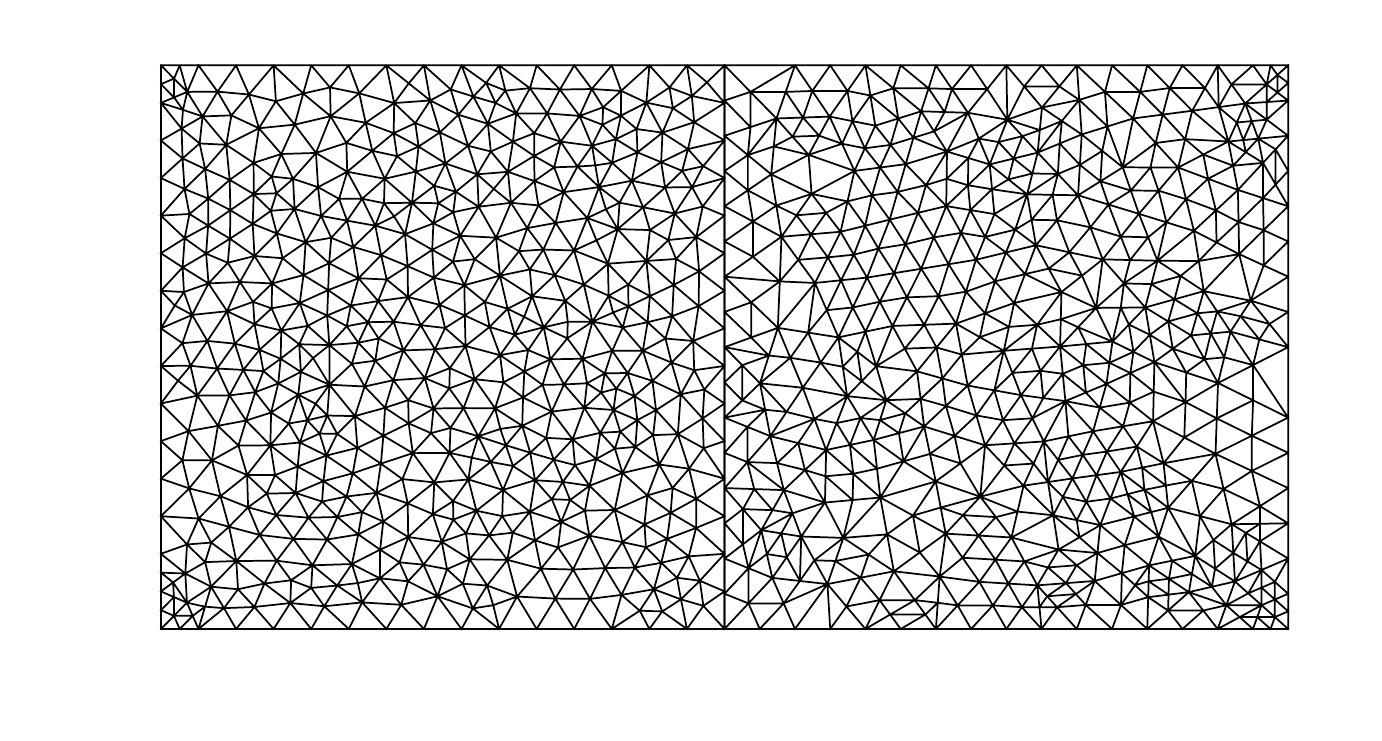}
    \includegraphics[width=0.49\textwidth]{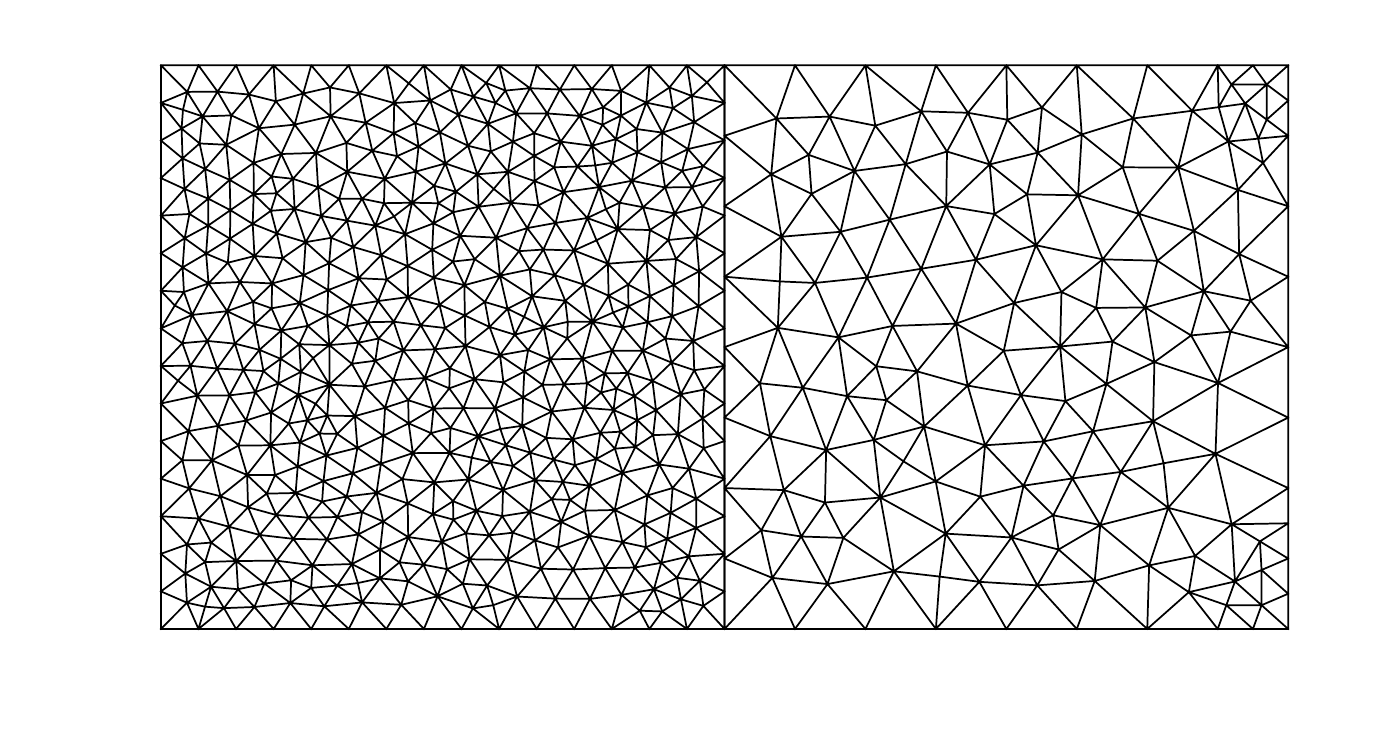}
    \includegraphics[width=0.49\textwidth]{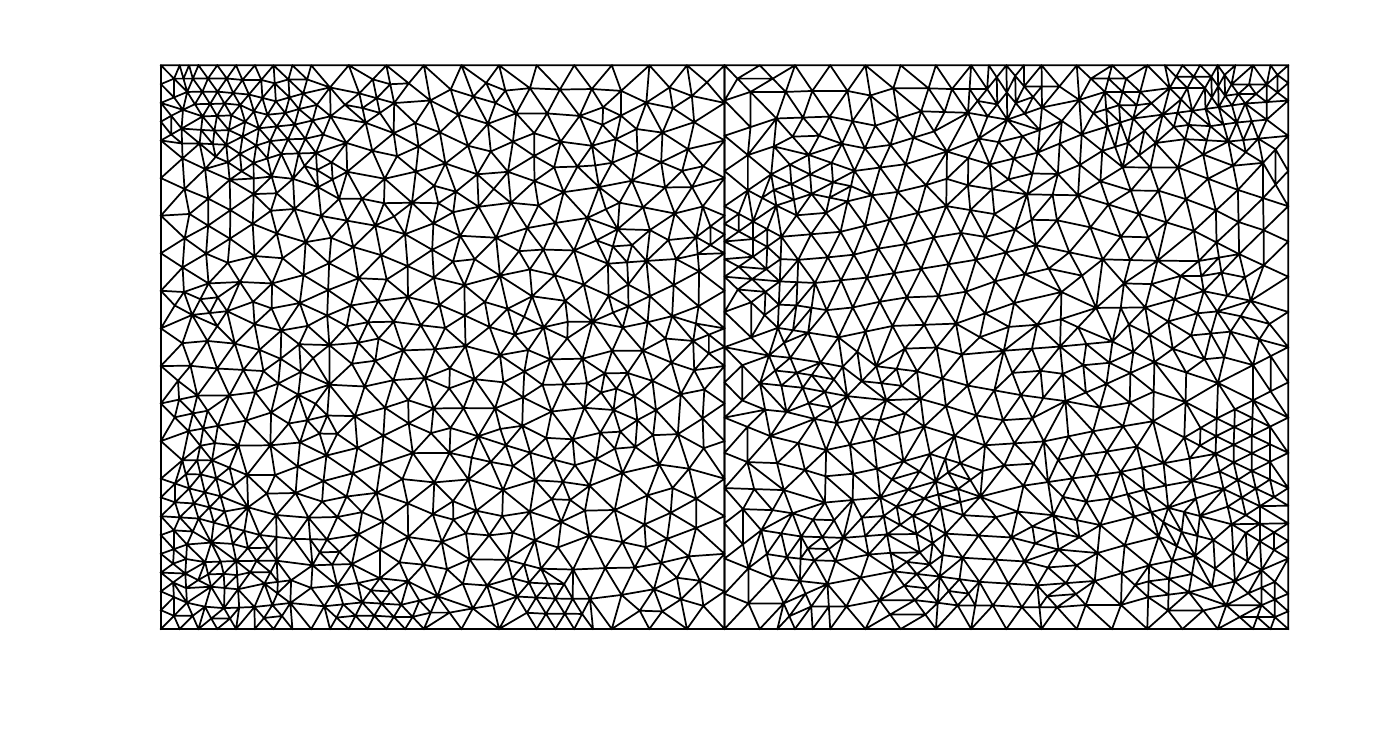}
    \includegraphics[width=0.49\textwidth]{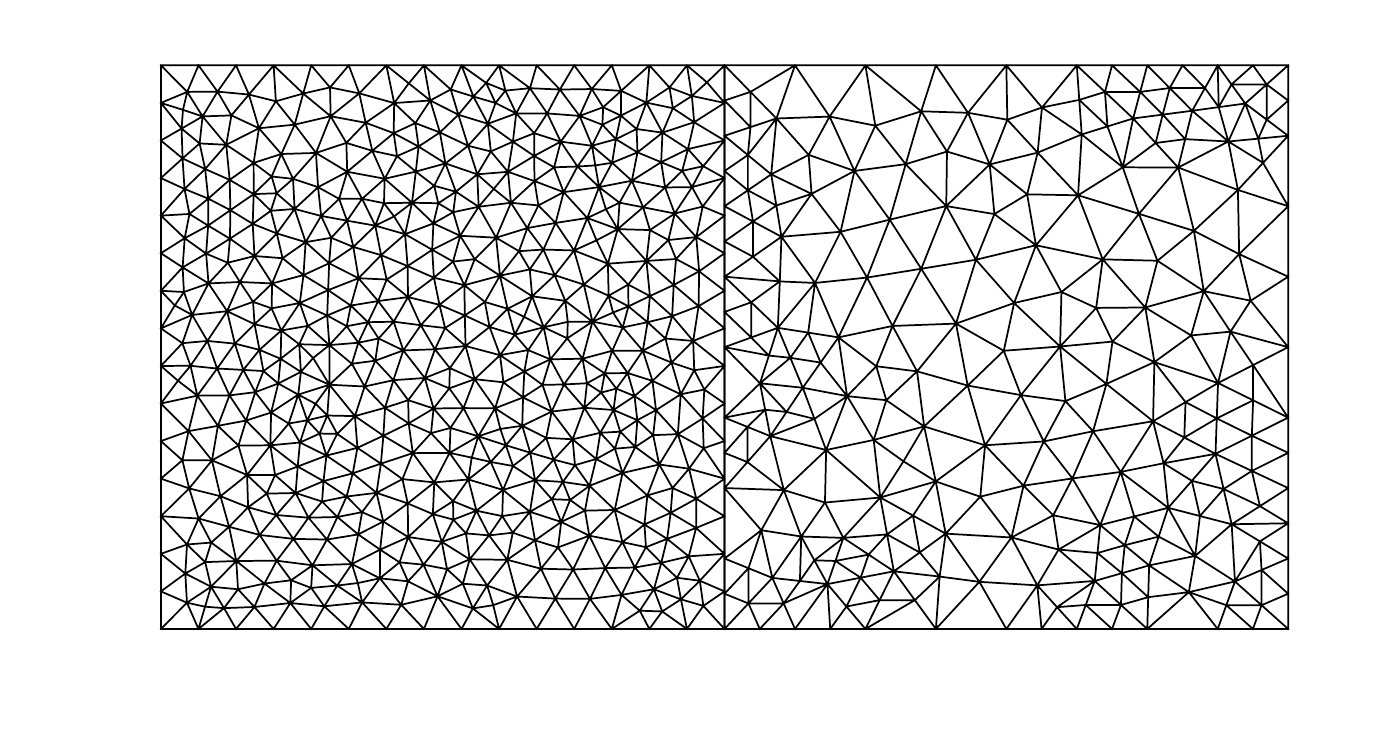}
    \includegraphics[width=0.49\textwidth]{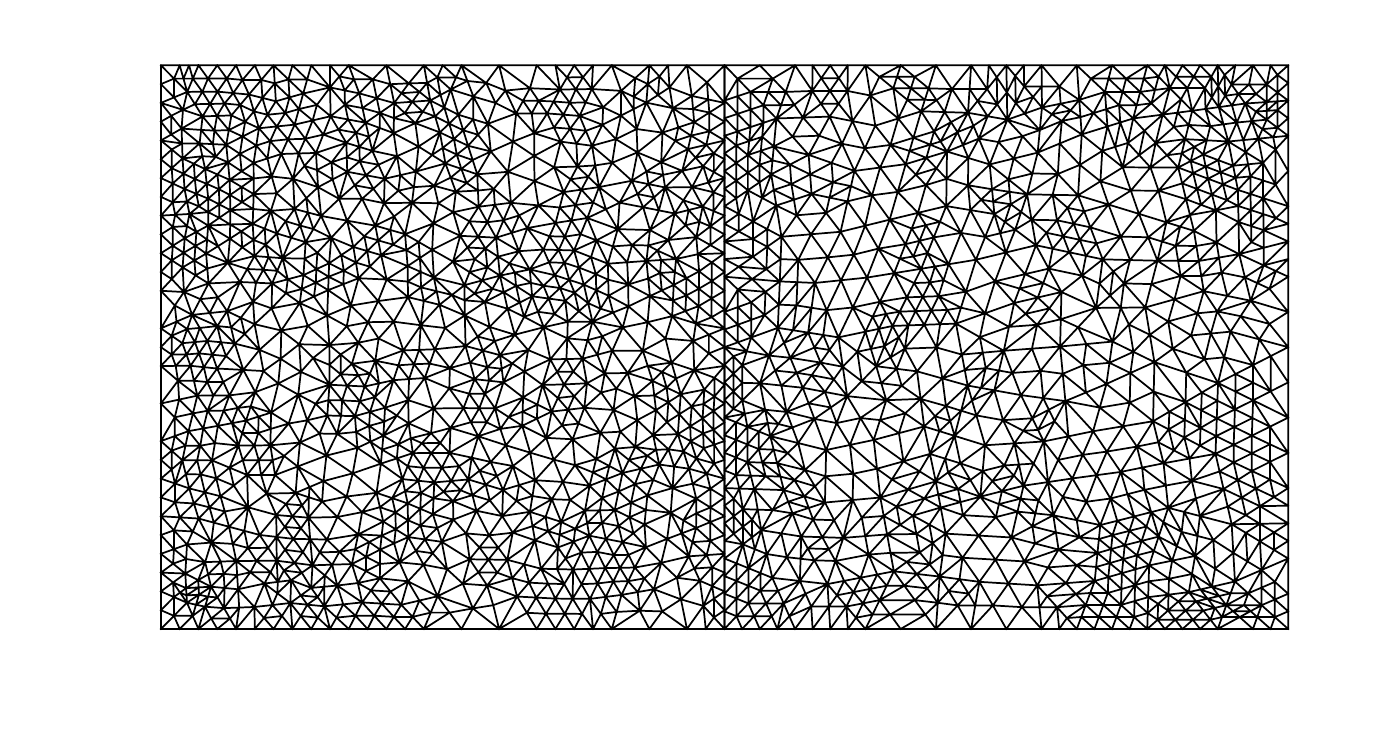}
    \includegraphics[width=0.49\textwidth]{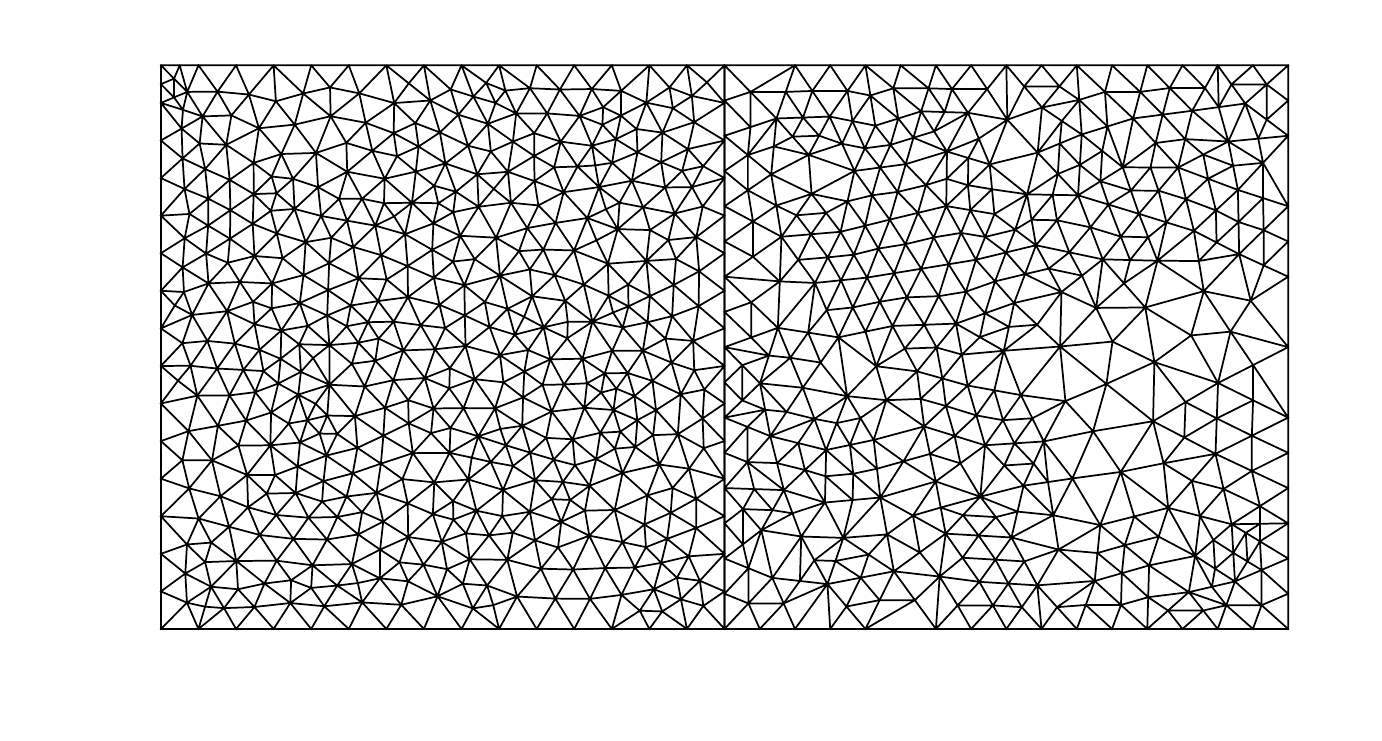}
    \includegraphics[width=0.49\textwidth]{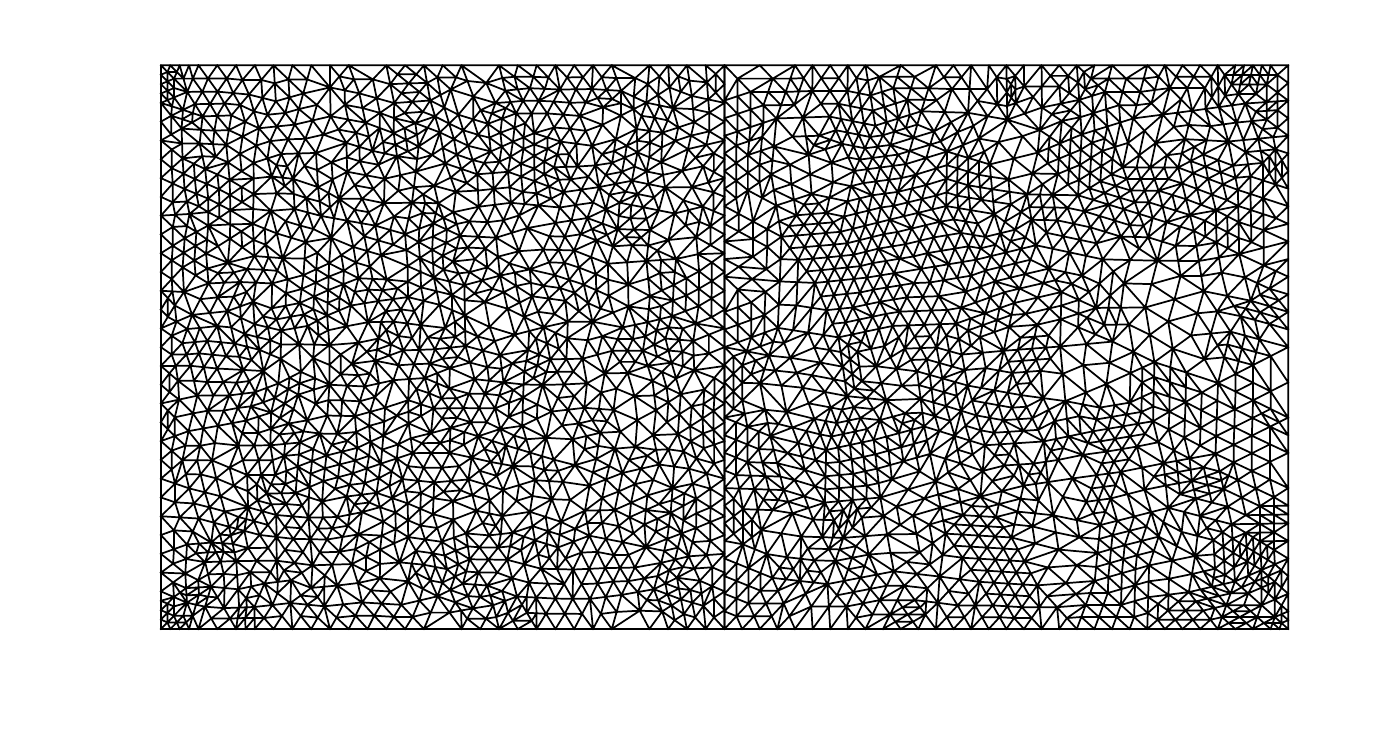}
    \caption{The sequence of adaptive meshes with $k_1=k_2=1$.}
    \label{fig:1}
\end{figure}

\begin{figure}[h!]
    \includegraphics[width=0.49\textwidth]{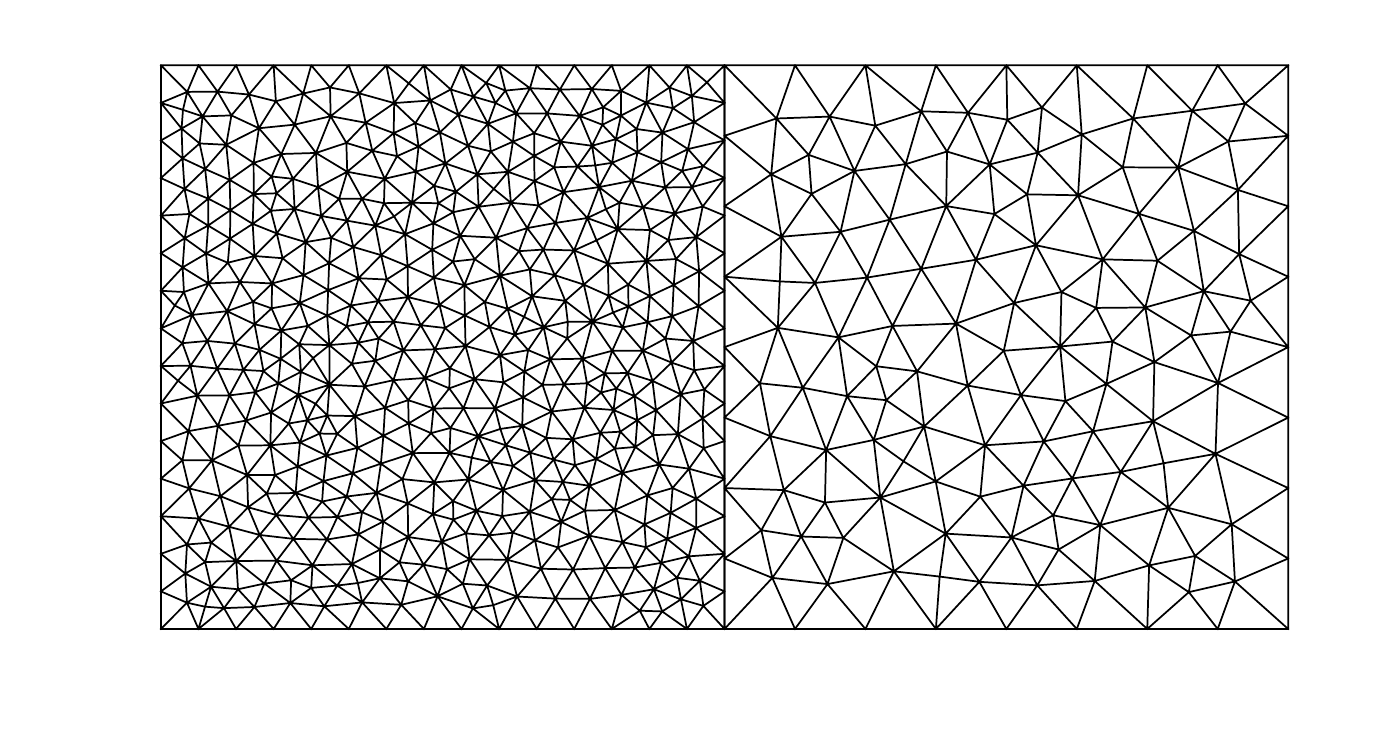}
    \includegraphics[width=0.49\textwidth]{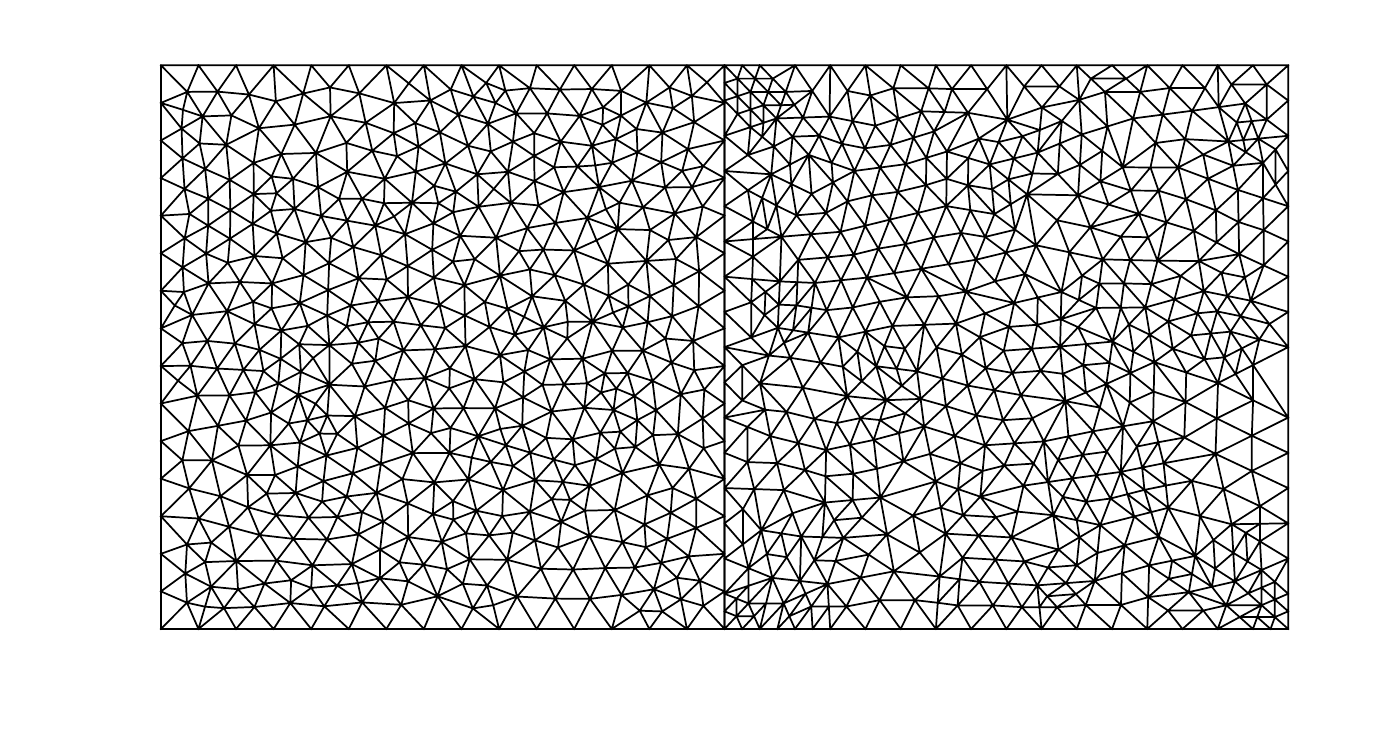}
    \includegraphics[width=0.49\textwidth]{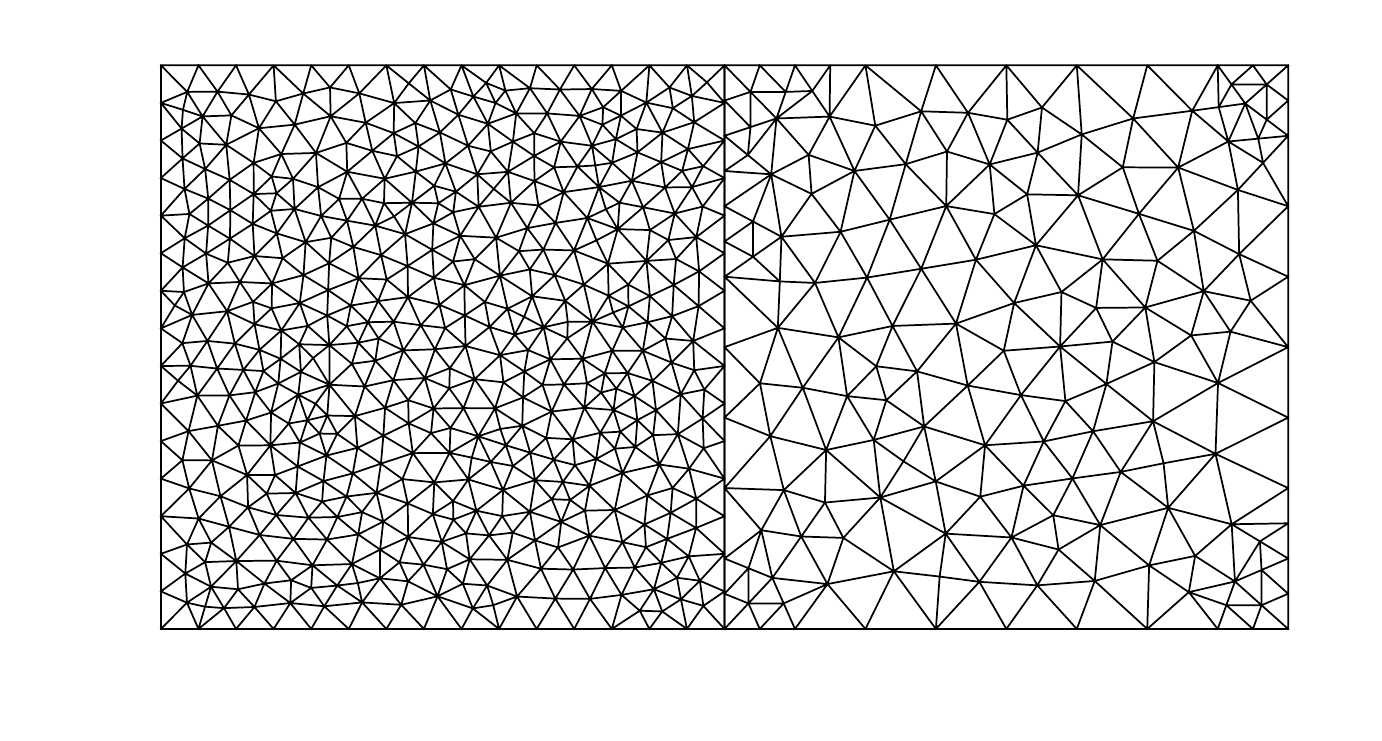}
    \includegraphics[width=0.49\textwidth]{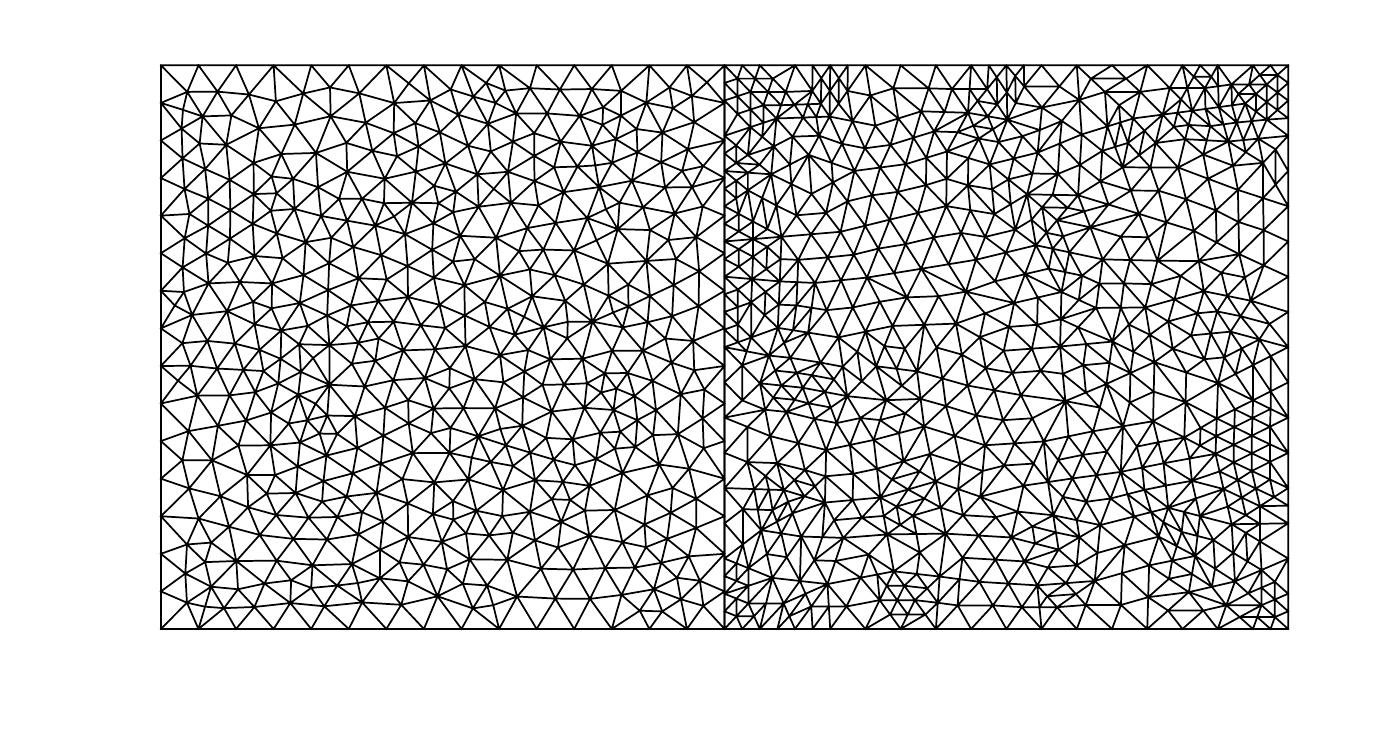}
    \includegraphics[width=0.49\textwidth]{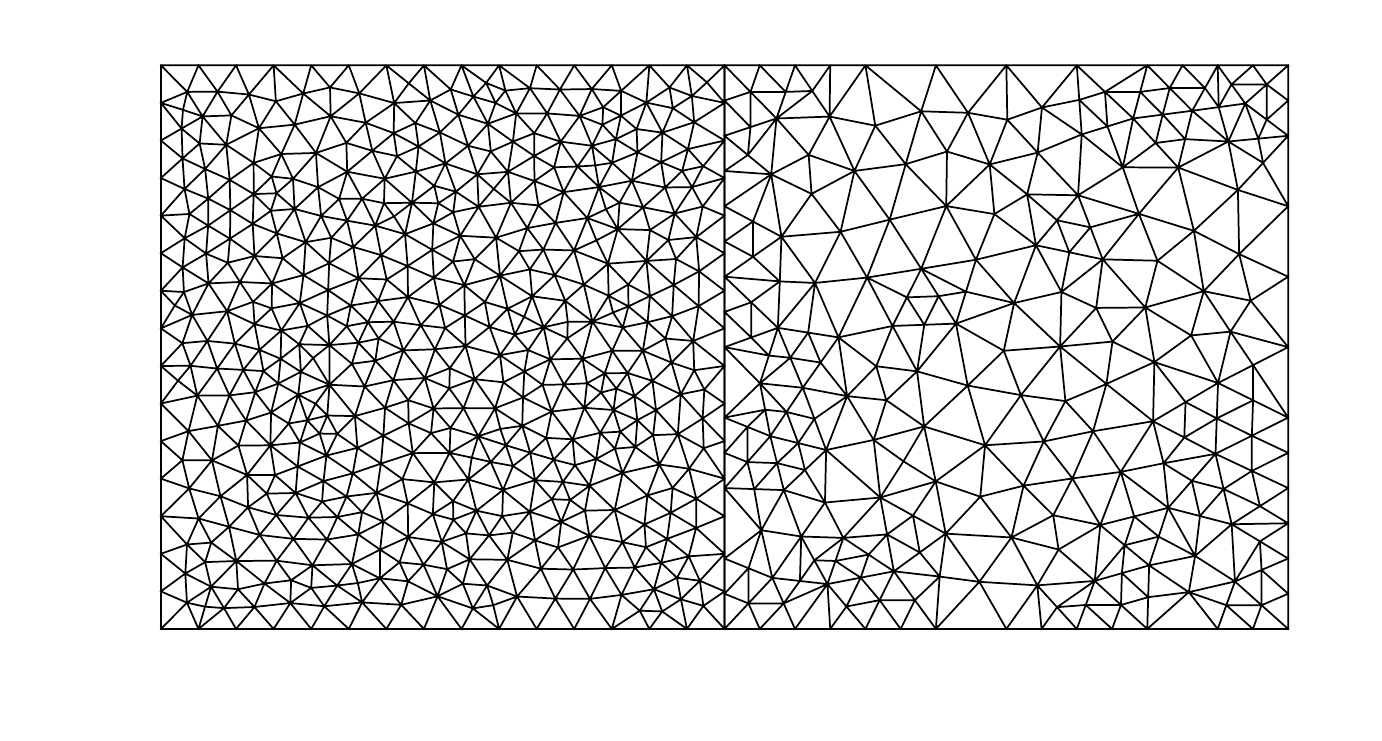}
    \includegraphics[width=0.49\textwidth]{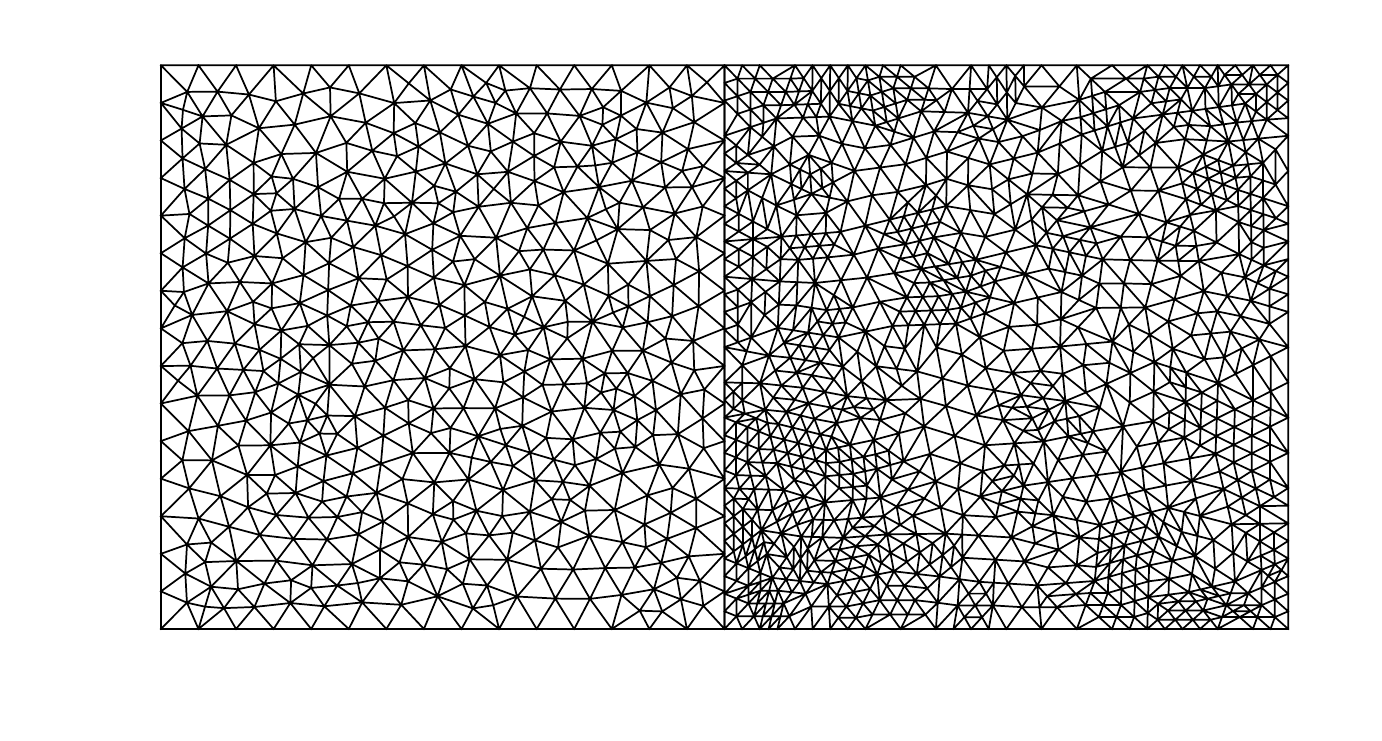}
    \includegraphics[width=0.49\textwidth]{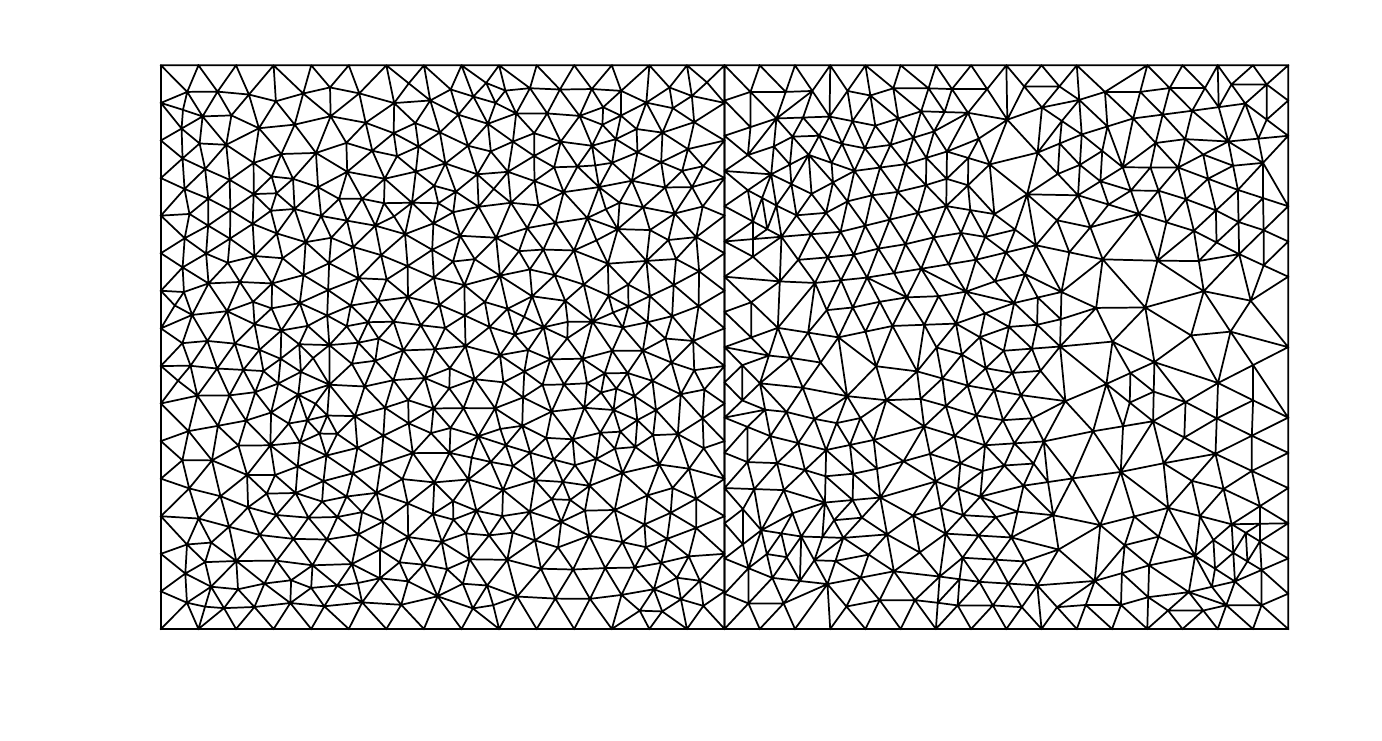}
    \includegraphics[width=0.49\textwidth]{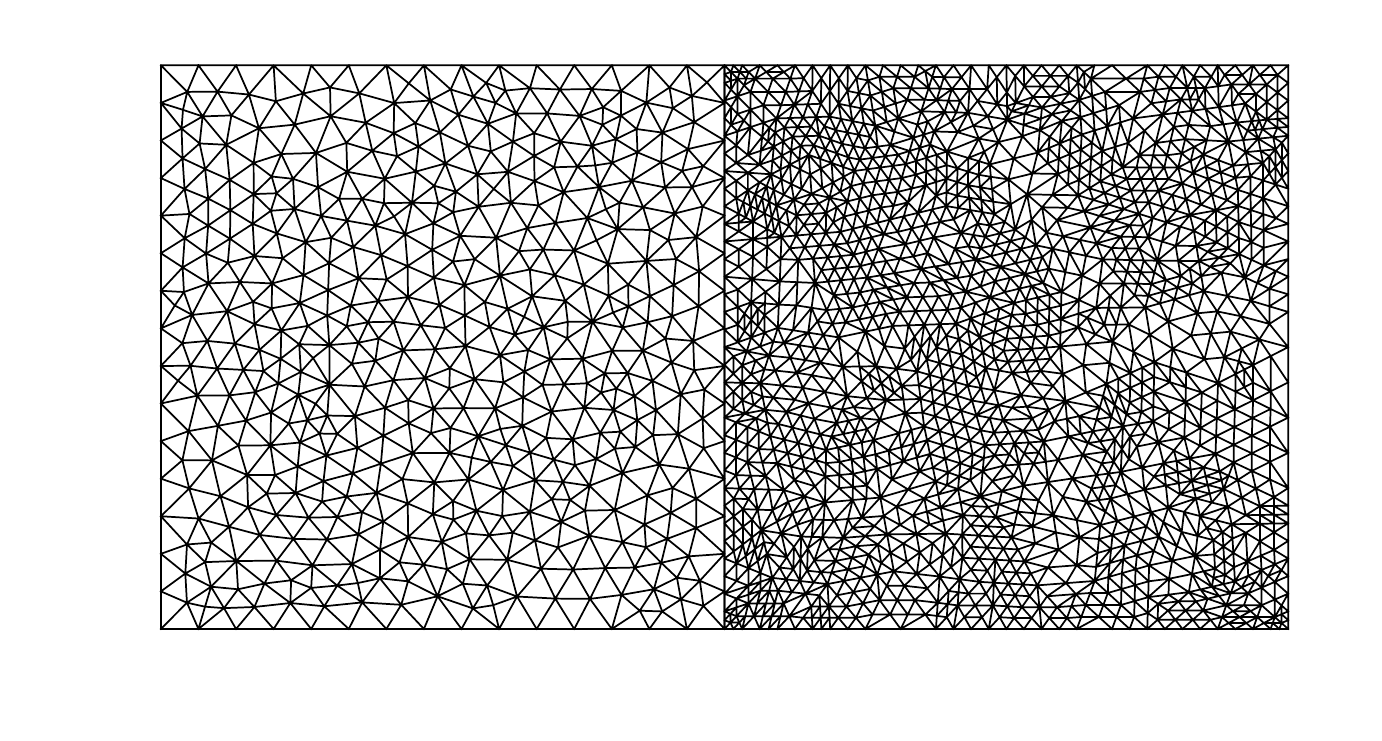}
    \caption{The sequence of adaptive meshes with $k_1=10$ and $k_2=0.1$.}
    \label{fig:2}
\end{figure}

\begin{figure}[h!]
    \includegraphics[width=0.49\textwidth]{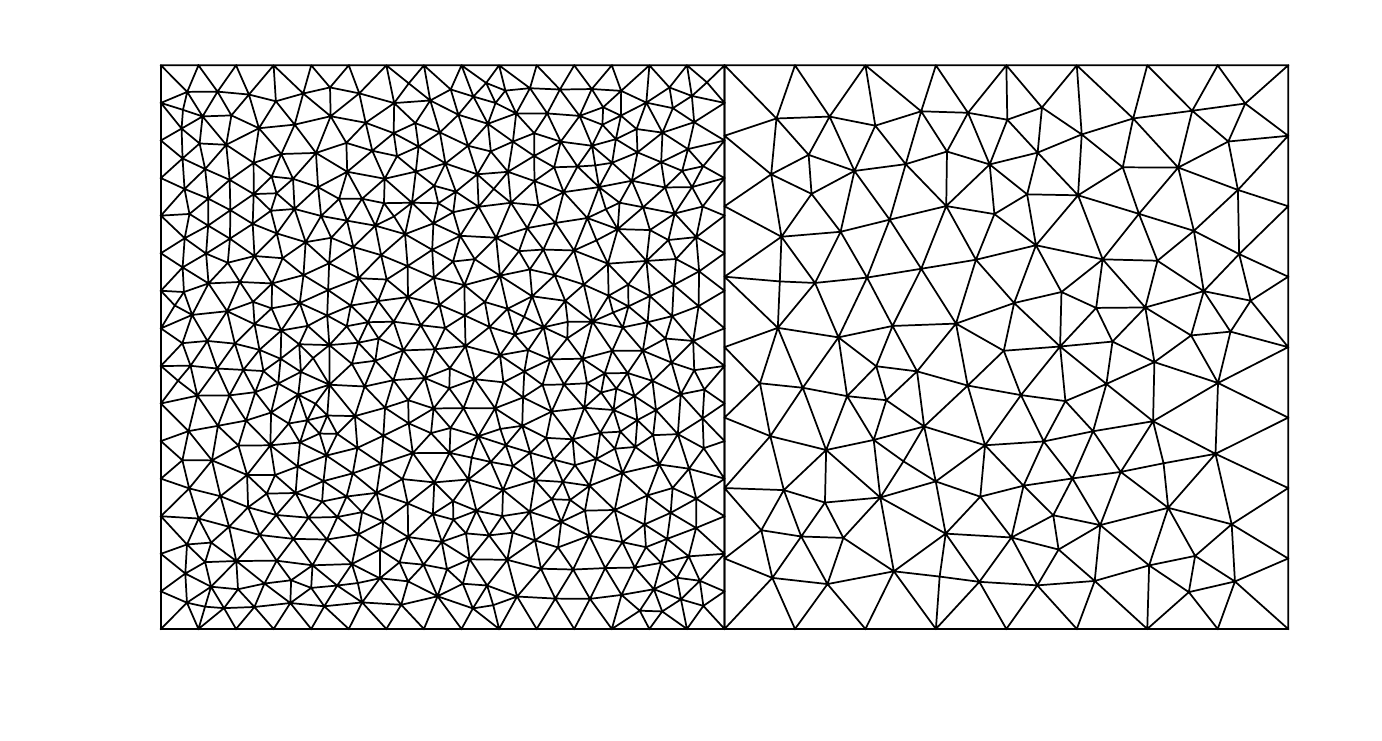}
    \includegraphics[width=0.49\textwidth]{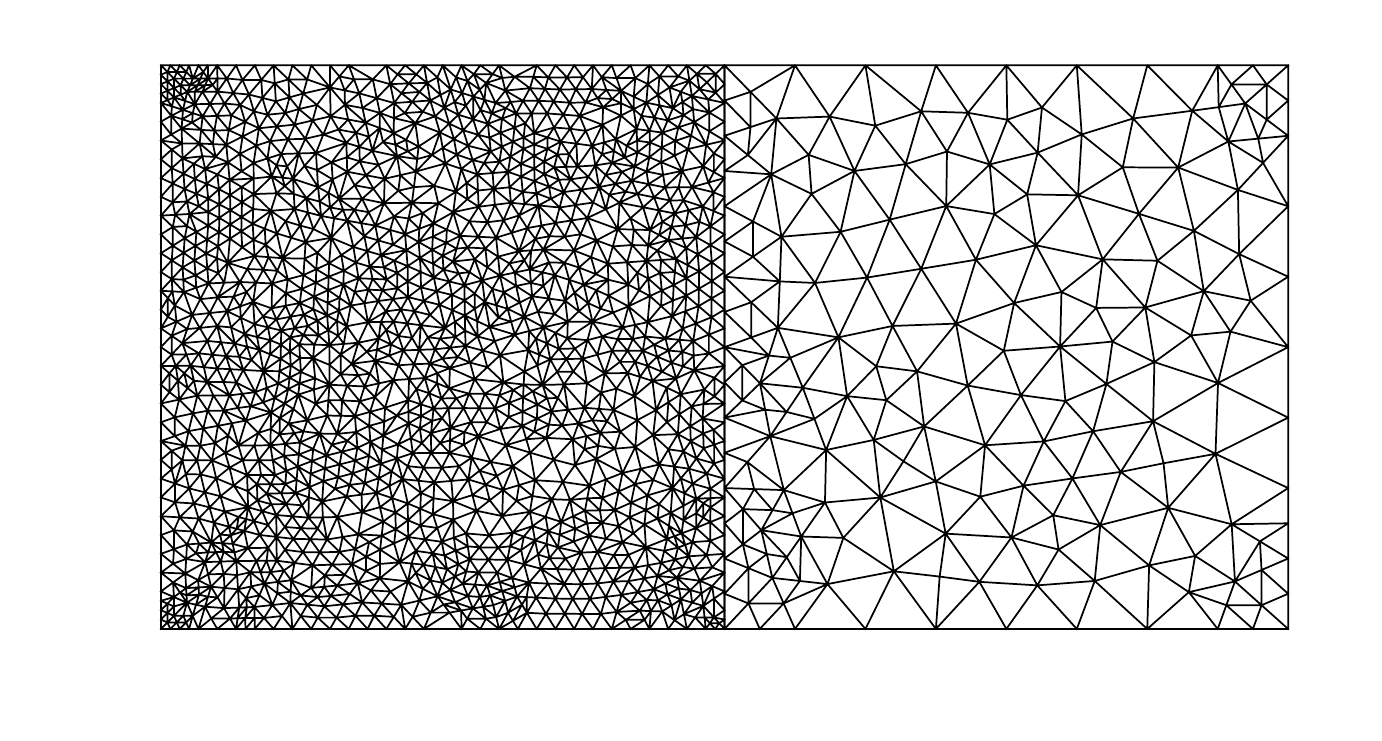}
    \includegraphics[width=0.49\textwidth]{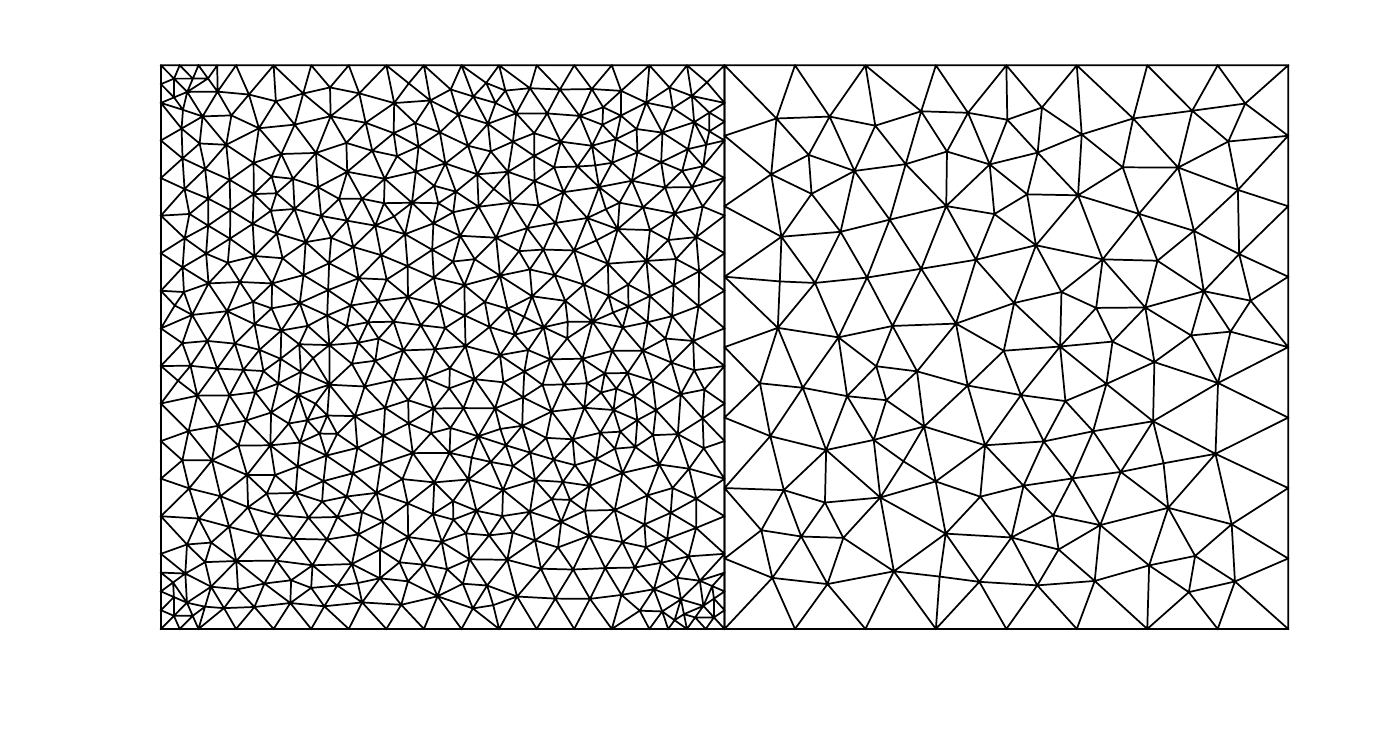}
    \includegraphics[width=0.49\textwidth]{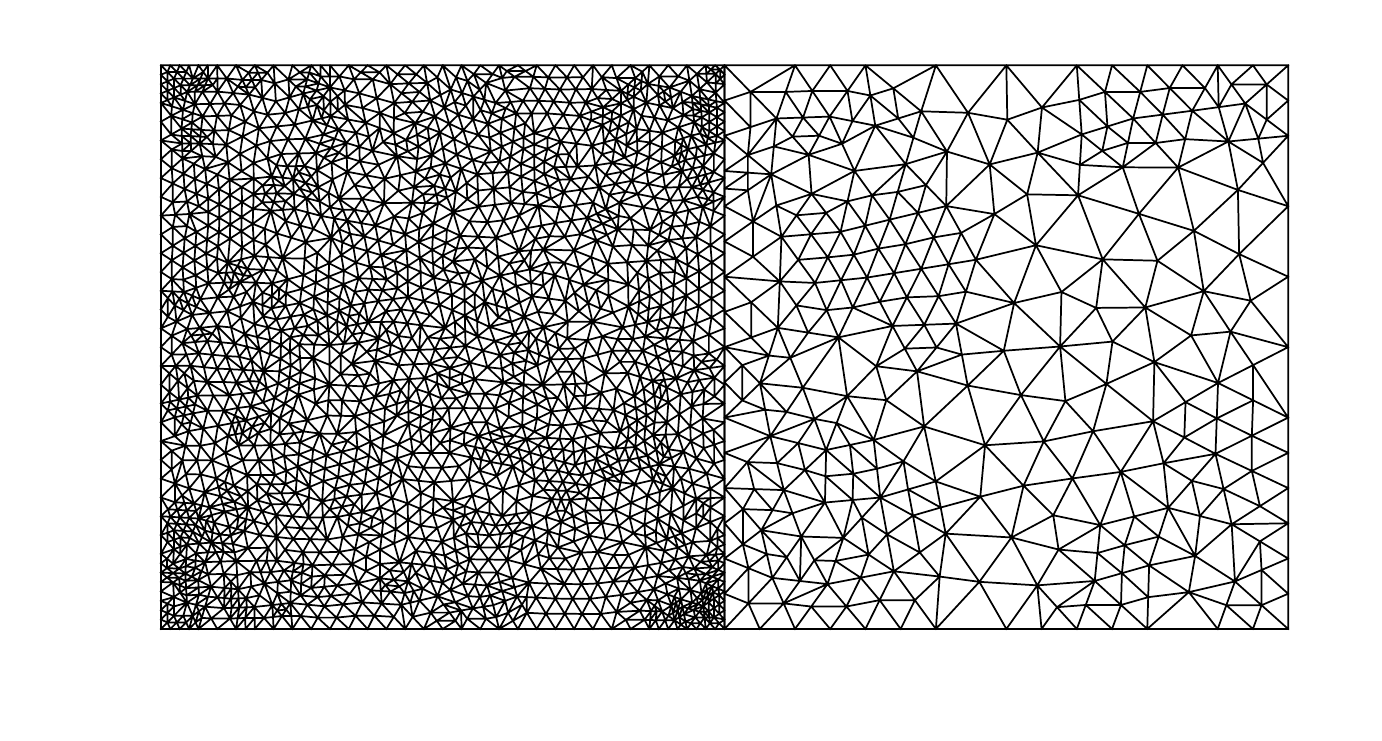}
    \includegraphics[width=0.49\textwidth]{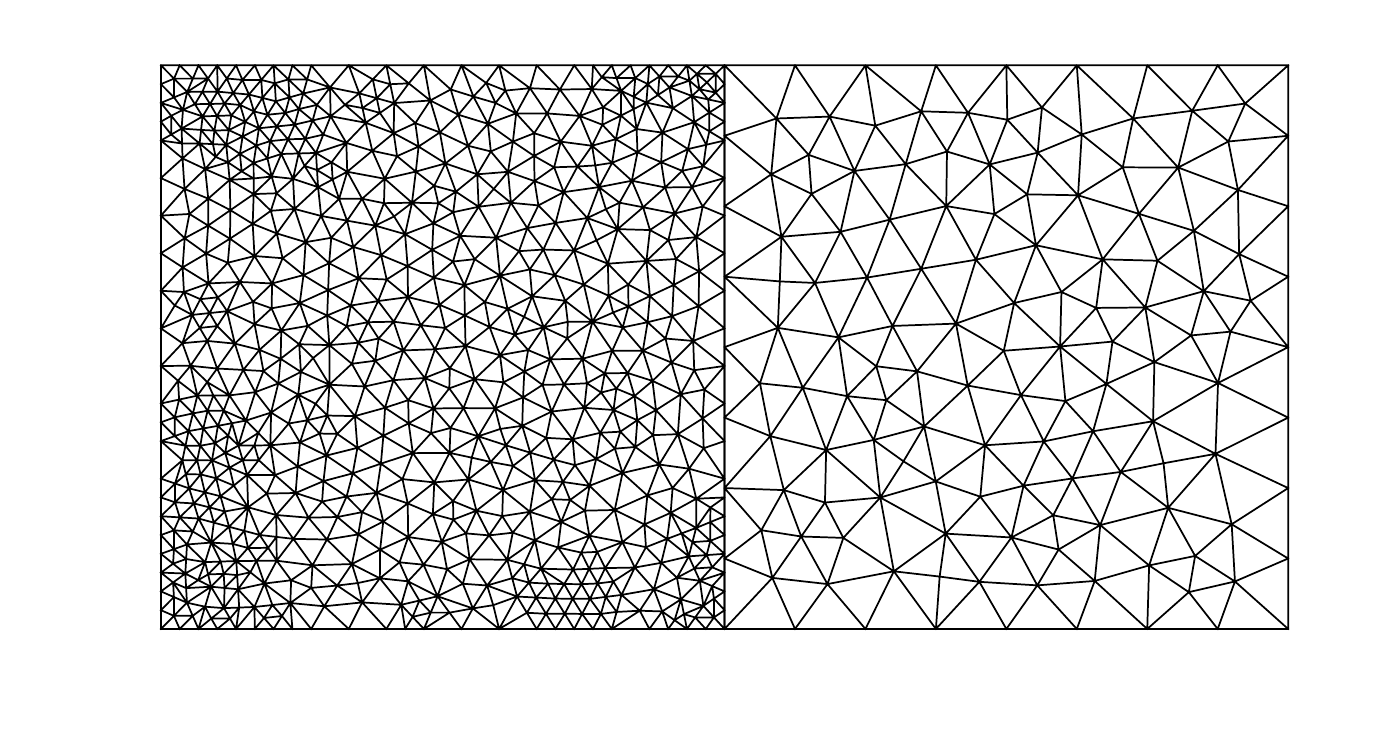}
    \includegraphics[width=0.49\textwidth]{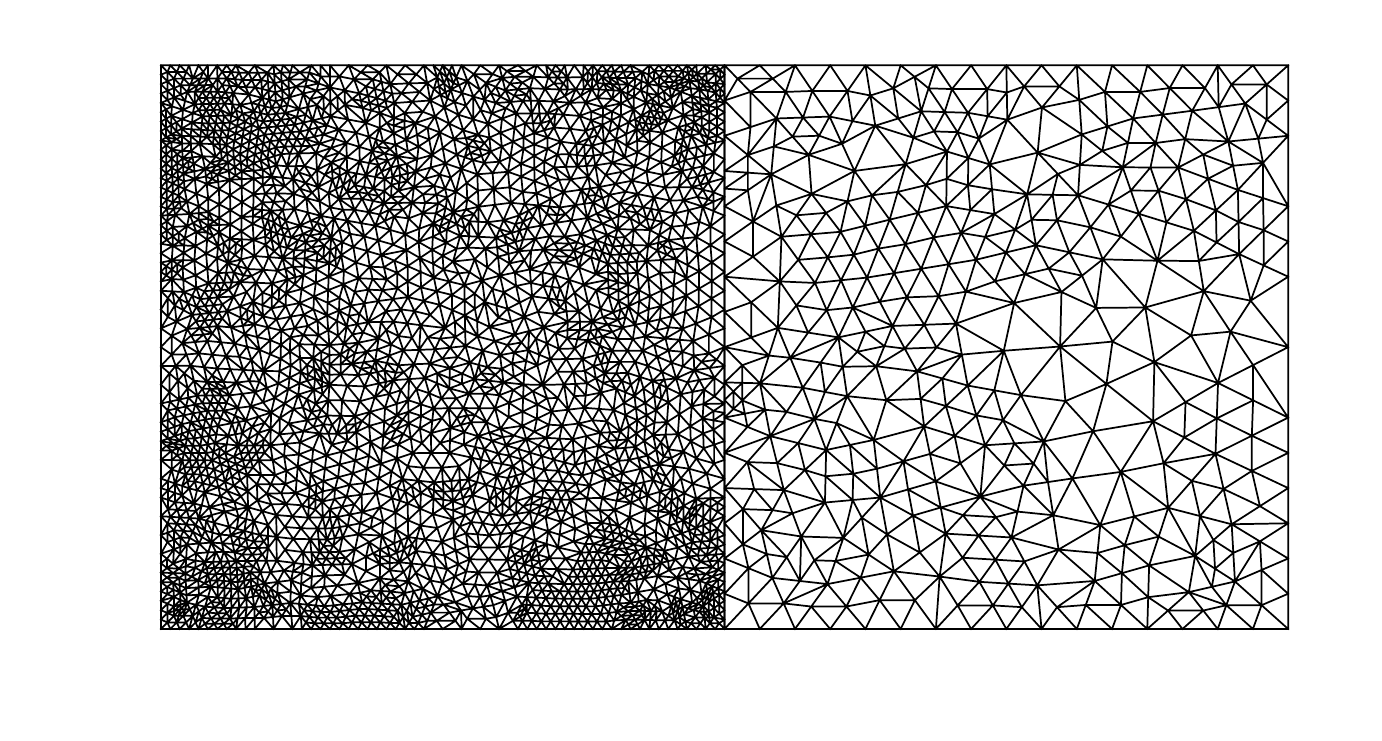}
    \includegraphics[width=0.49\textwidth]{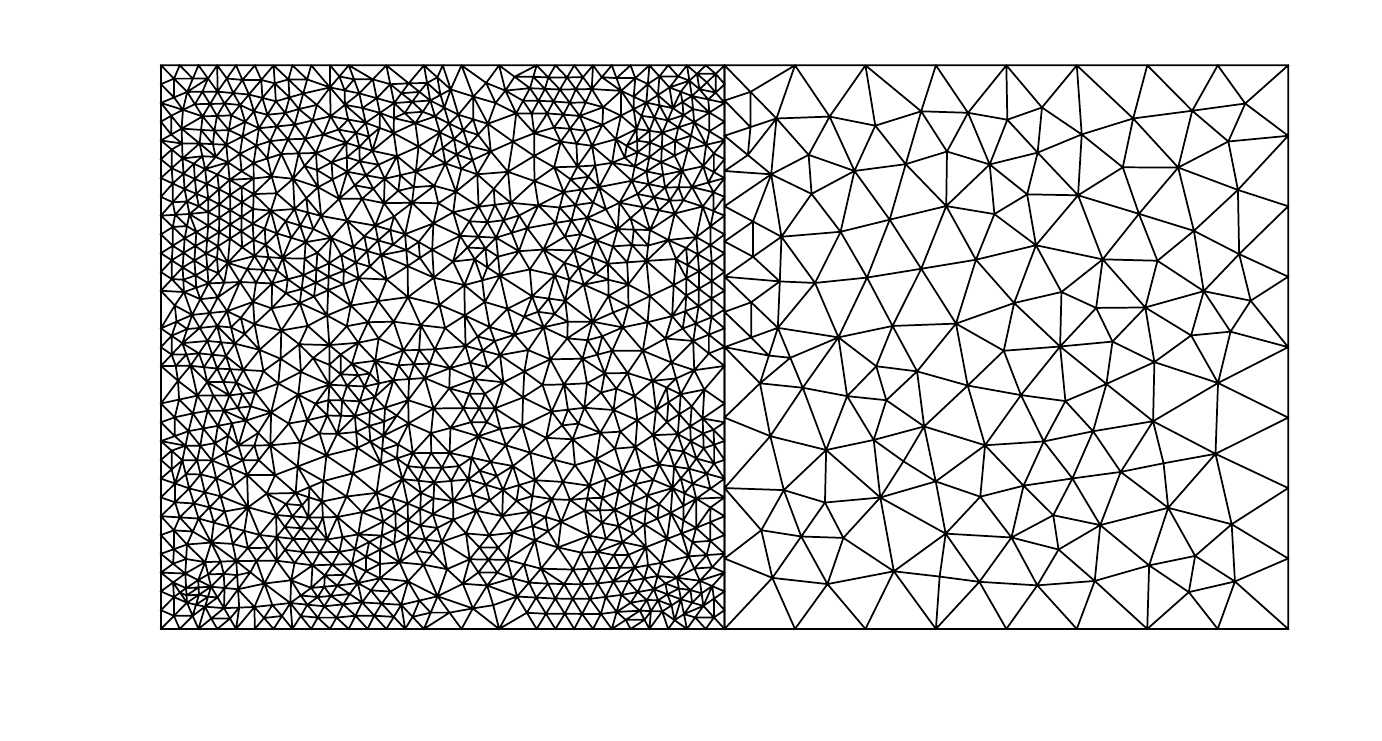}
    \includegraphics[width=0.49\textwidth]{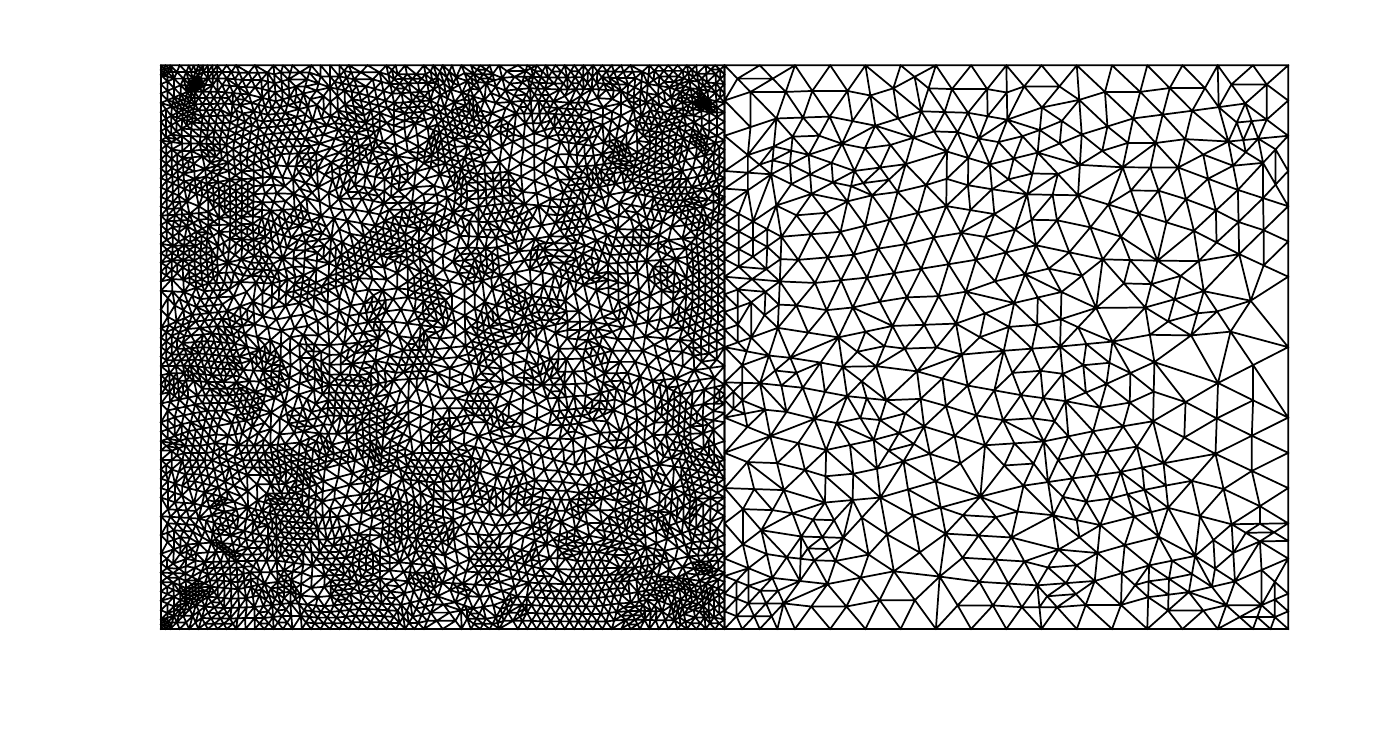}
    \caption{The sequence of adaptive meshes with $k_1=0.1$ and $k_2=10$.}
    \label{fig:3}
\end{figure}

\begin{figure}[h!]
    \includegraphics[width=0.32\textwidth]{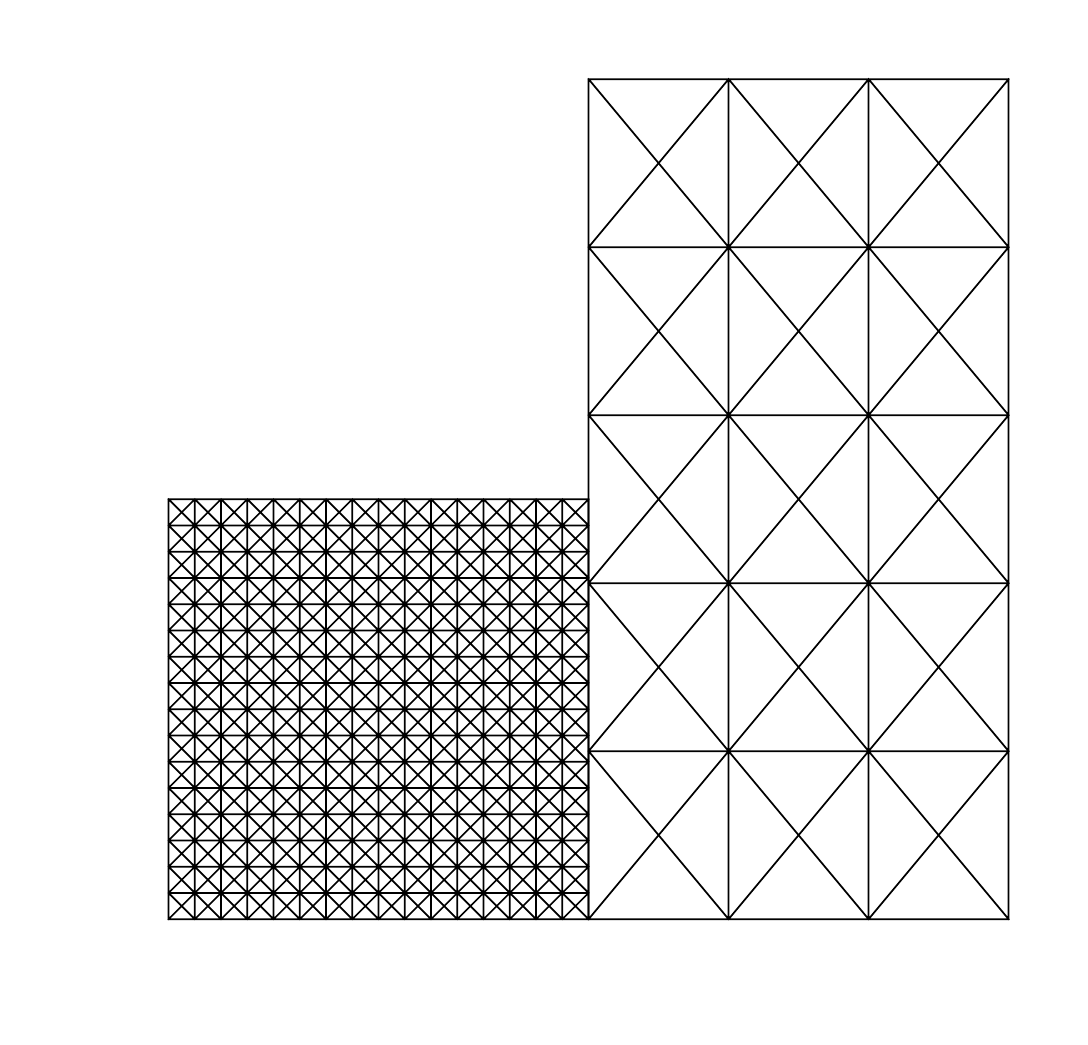}
    \includegraphics[width=0.32\textwidth]{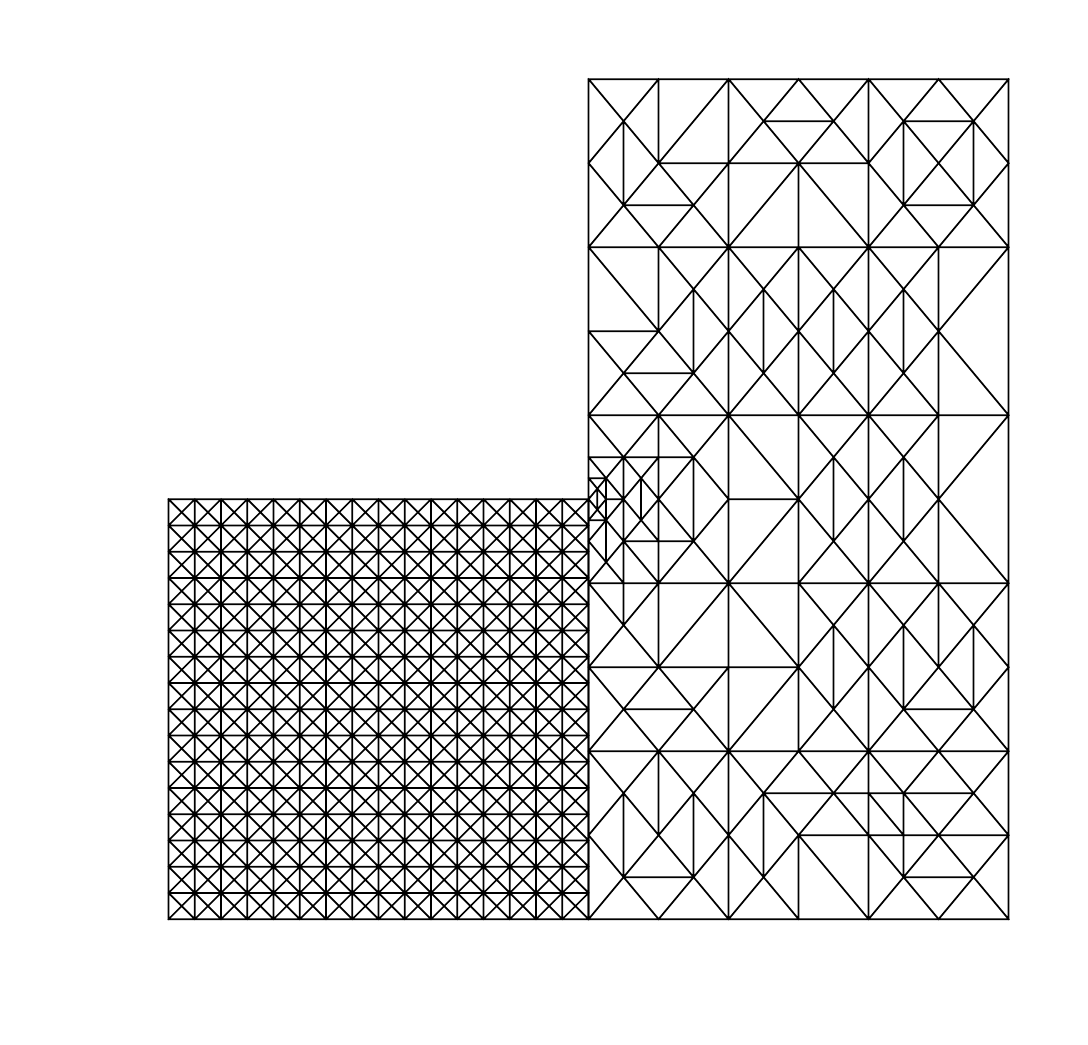}
    \includegraphics[width=0.32\textwidth]{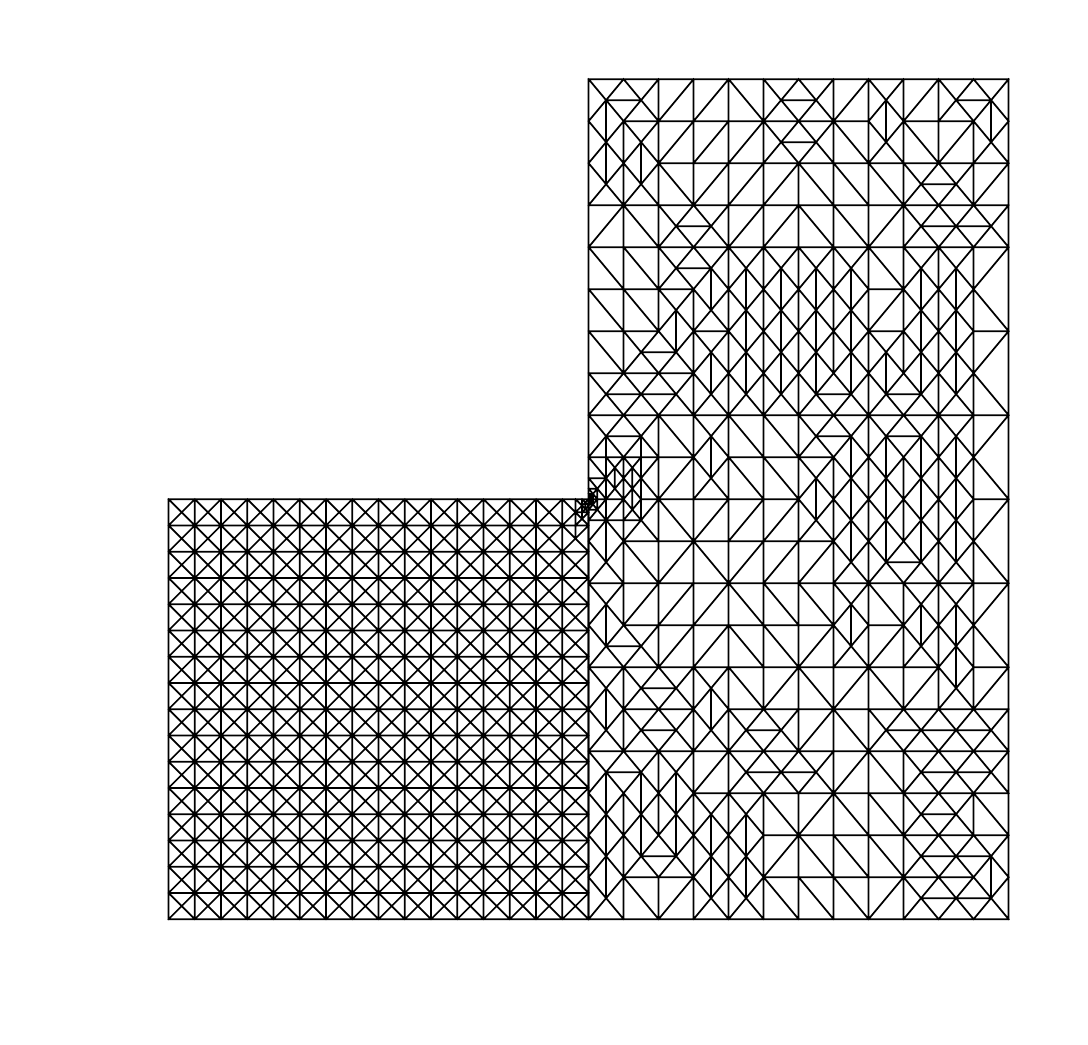}
    \includegraphics[width=0.32\textwidth]{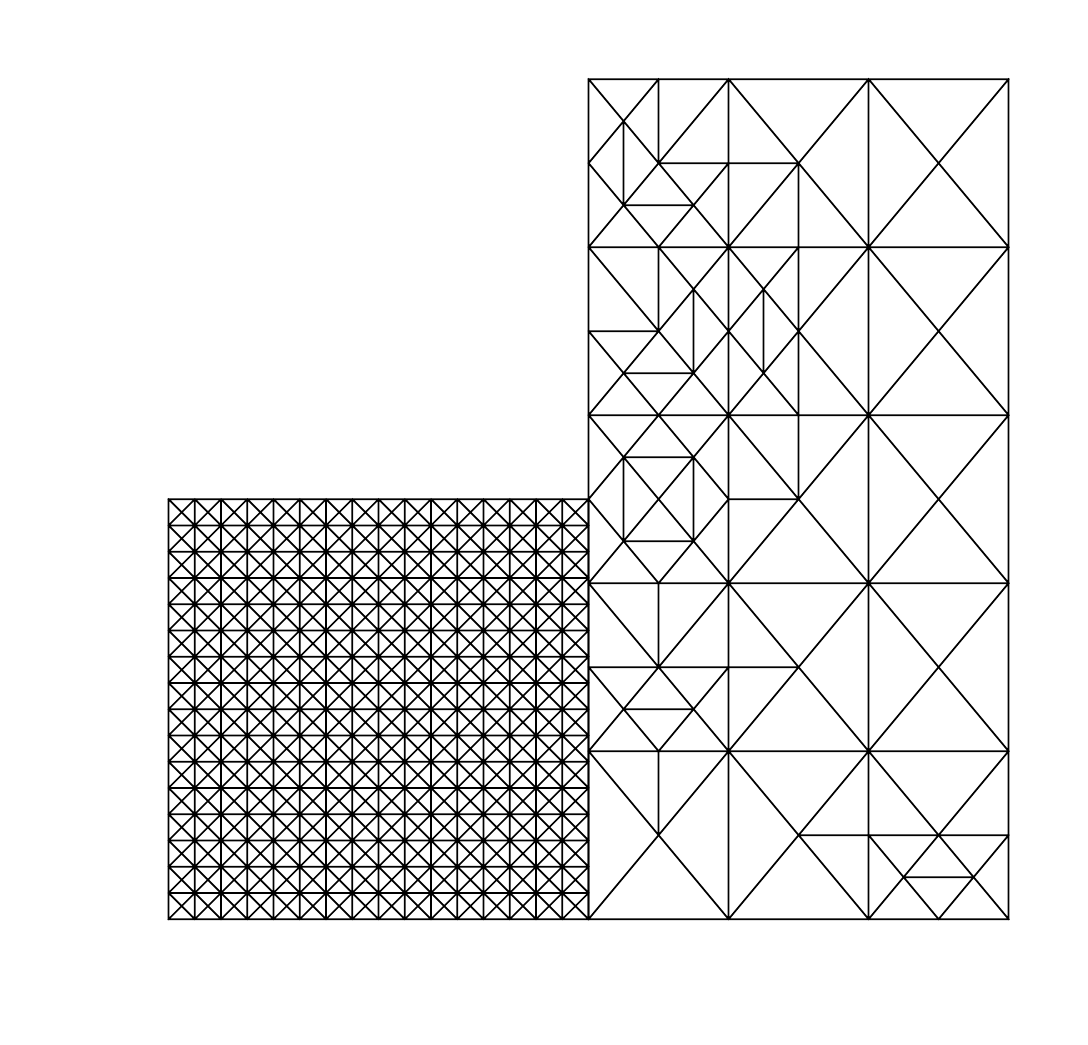}
    \includegraphics[width=0.32\textwidth]{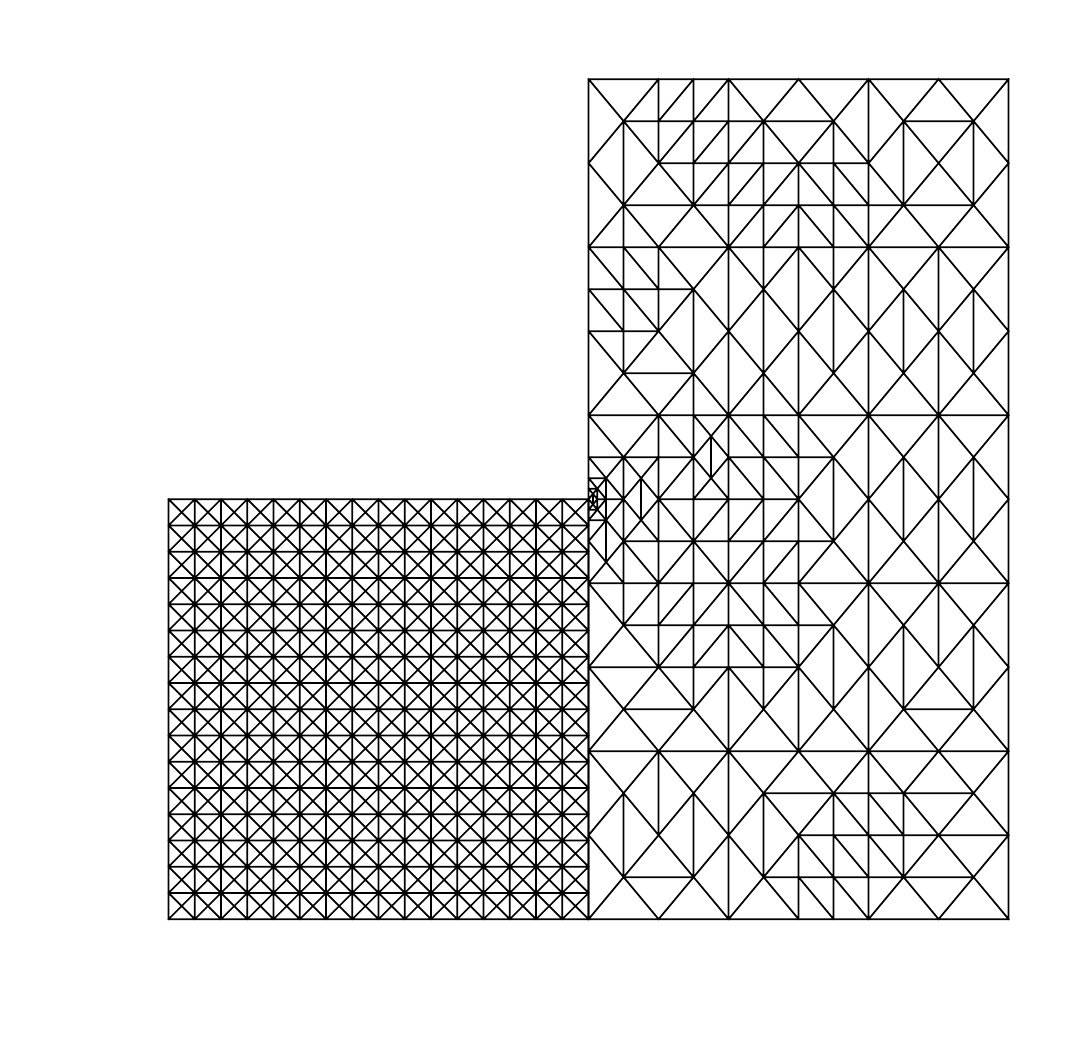}
    \includegraphics[width=0.32\textwidth]{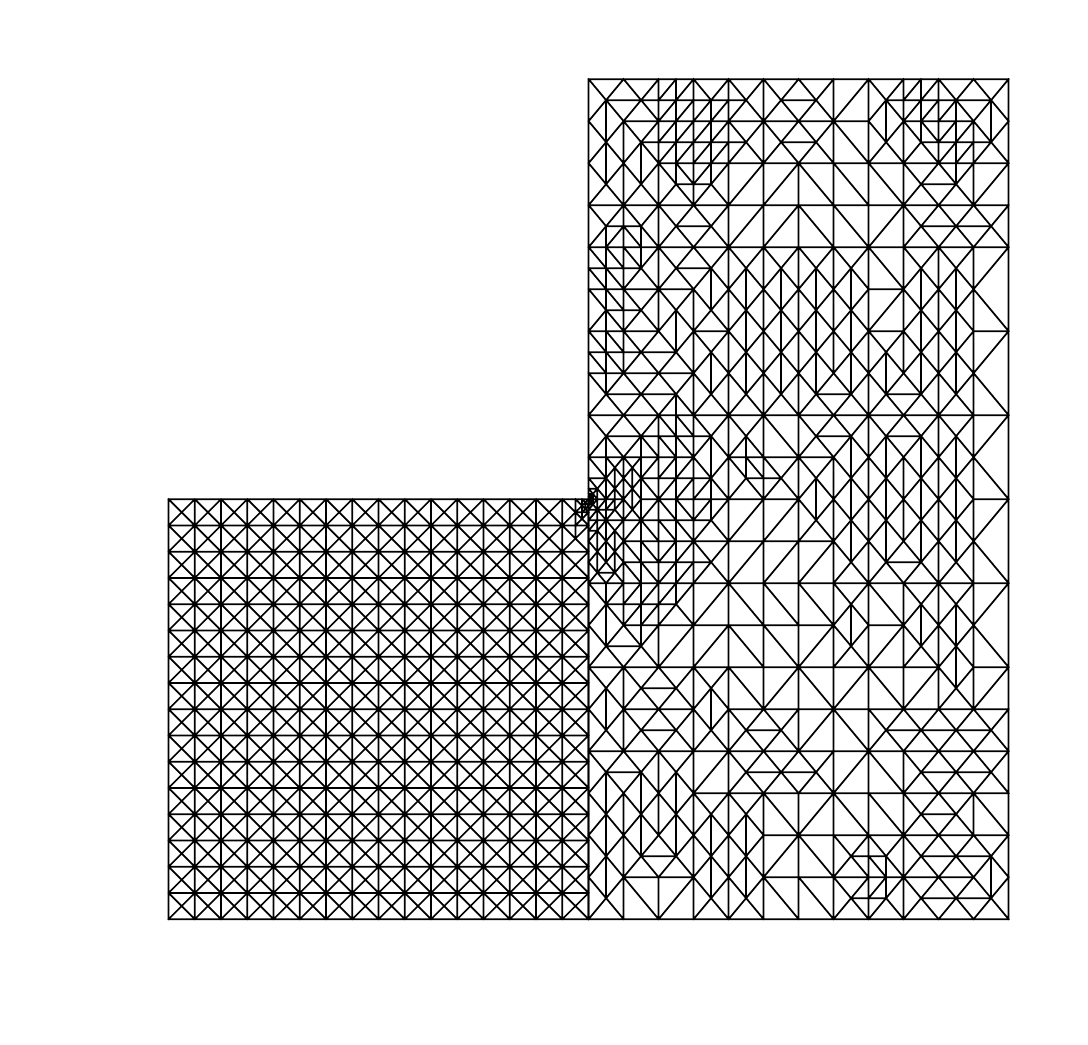}
    \includegraphics[width=0.32\textwidth]{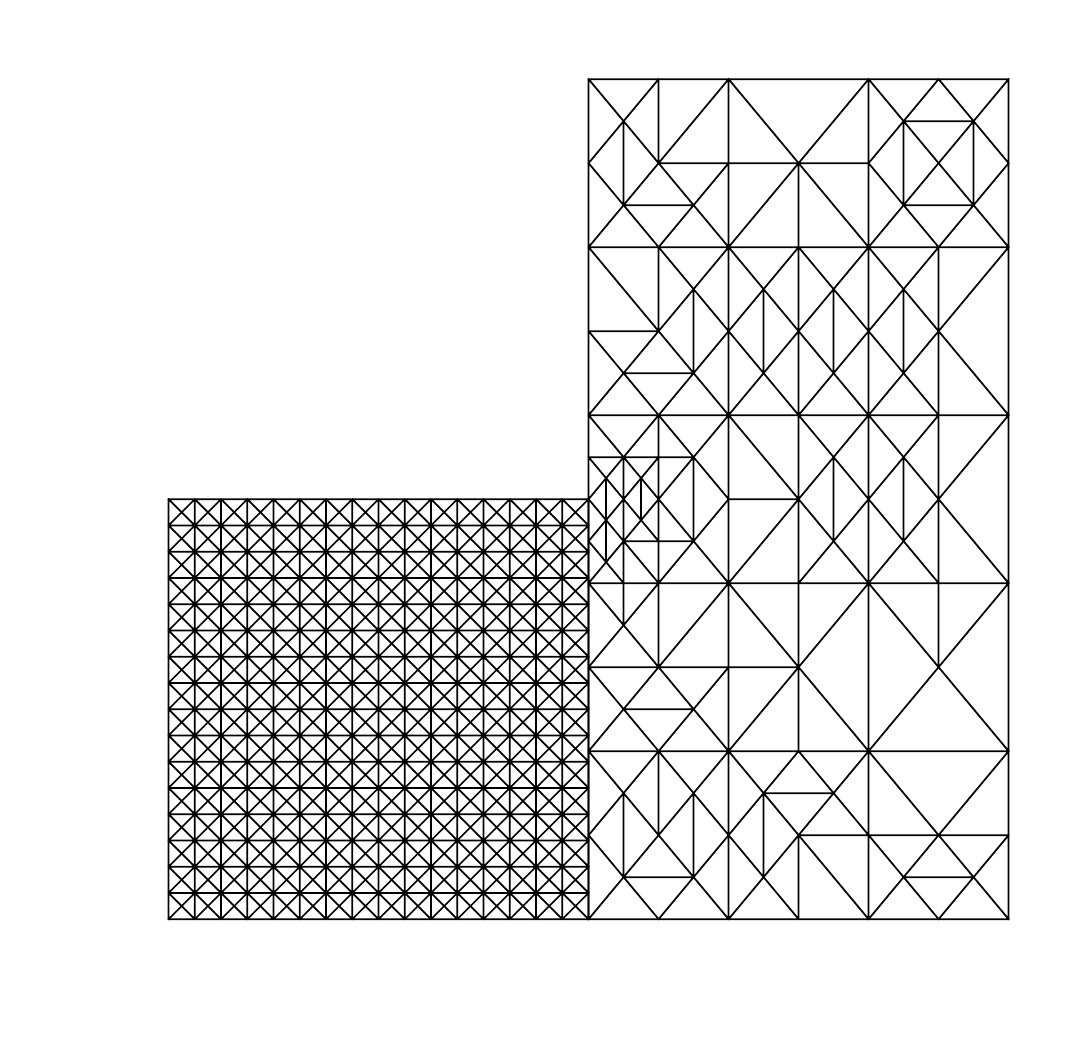}
    \includegraphics[width=0.32\textwidth]{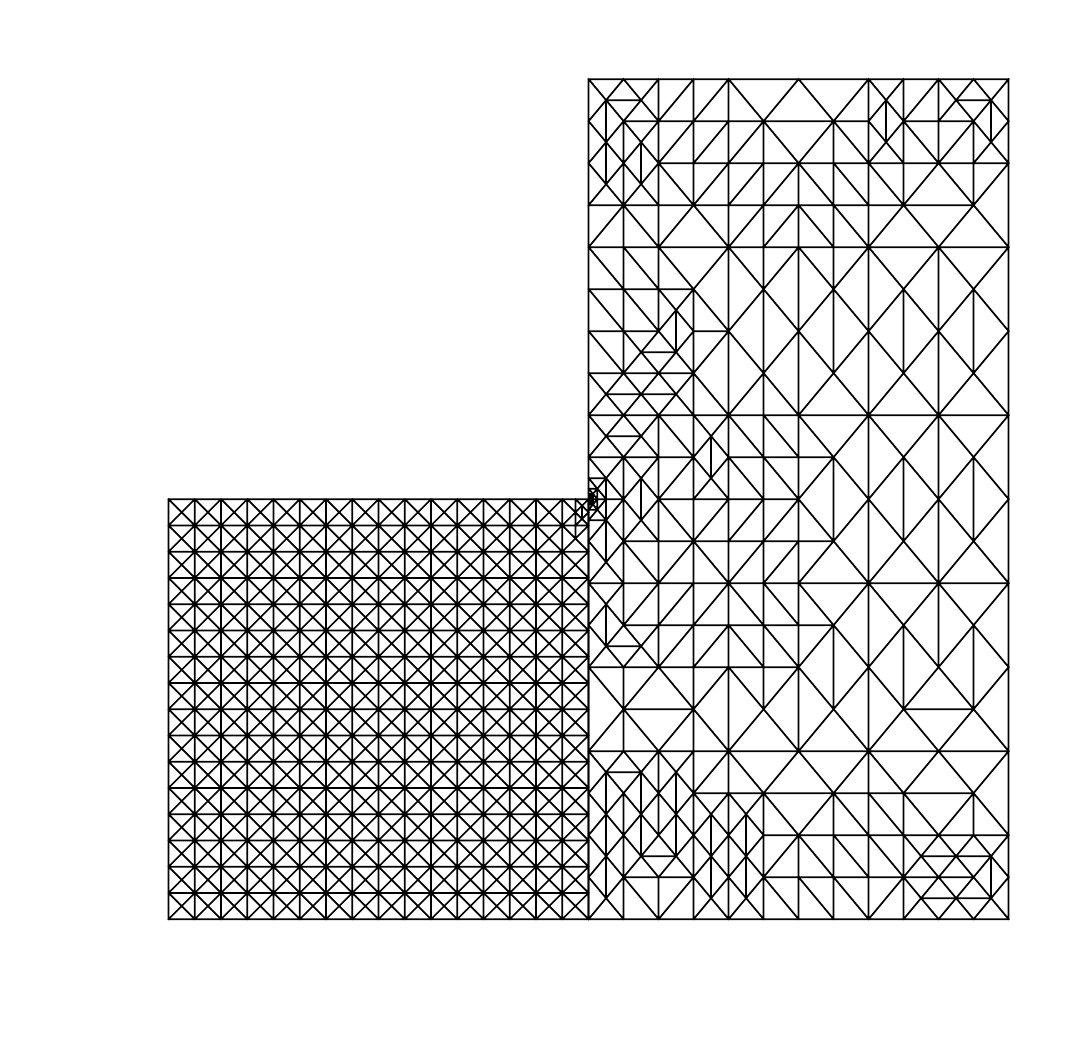}
    \includegraphics[width=0.32\textwidth]{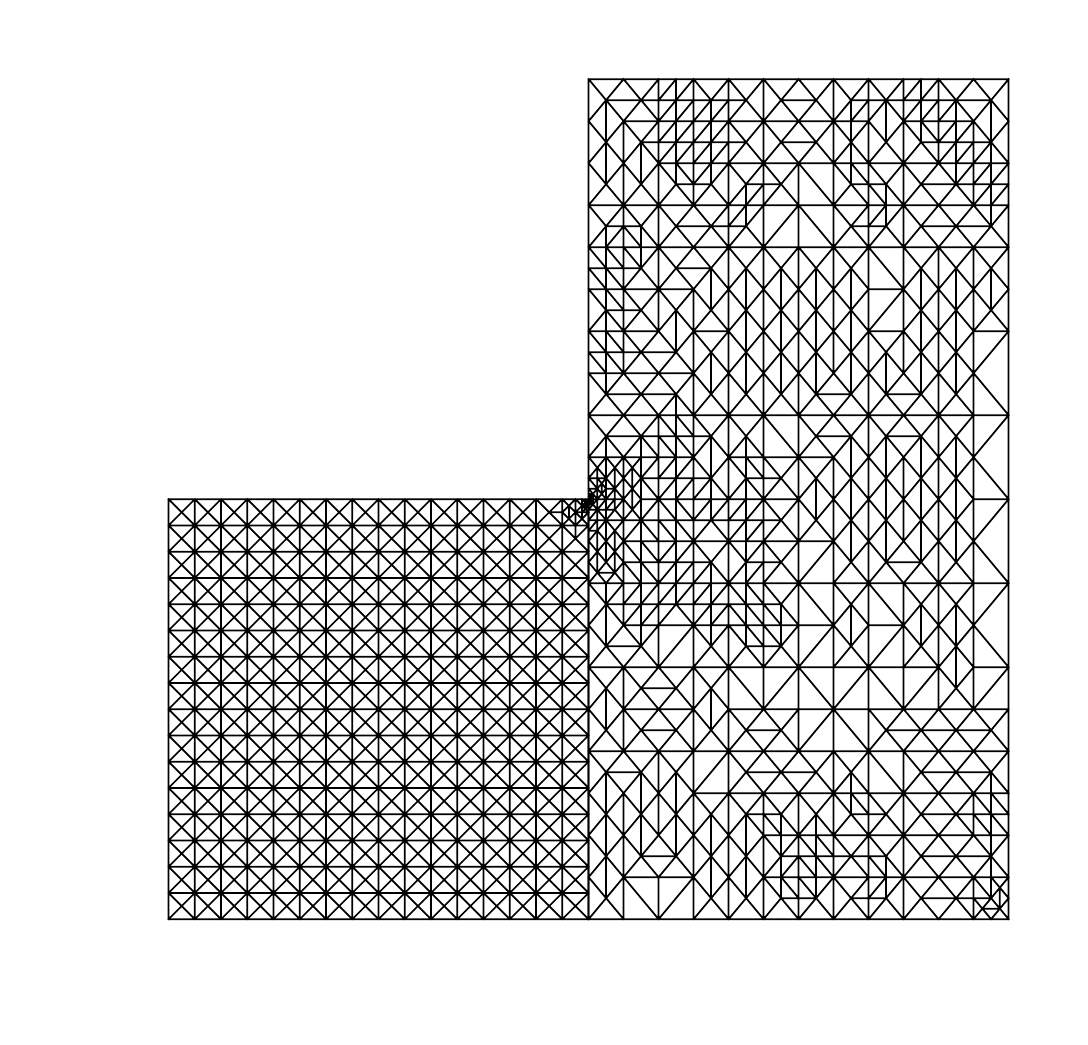}
    \caption{The sequence of adaptive meshes on an L-shaped domain with $k_1=k_2=1$.}
    \label{fig:lshaped}
\end{figure}

\pgfplotstableread{
    ndofs eta
    584 0.497465283416
    618 0.408851575761
    663 0.318526652926
    684 0.293186481151
    741 0.253878856143
    832 0.218005000682
    989 0.179752260047
    1098 0.162830141809
    1192 0.152787466758
    1441 0.135627816236
    1562 0.129325601261
    2080 0.112639934106
    2411 0.104685288917
    3270 0.0920626530931
    4343 0.0797769237903
    5668 0.0680911298602
    6064 0.0655813145573
    8153 0.0568470058191
    9357 0.0531890413469
}\adaptive

\pgfplotstableread{
    ndofs eta
    584 0.497465283416
    2250 0.297390007594
    8834 0.178152204521
}\uniform

\begin{figure}
    \centering
    \begin{tikzpicture}[scale=0.8]
        \begin{axis}[
                xmode = log,
                ymode = log,
                xlabel = {$N$},
                ylabel = {$\eta$},
                grid = both
            ]
            \addplot table[x=ndofs,y=eta] {\adaptive};
            \addplot table[x=ndofs,y=eta] {\uniform};
            %\addplot table[x=ndofs,y=eta] {\adaptivepartial};
            %\addplot table[x=ndofs,y=eta] {\adaptivefull};
            \addplot+ [black, domain=2e3:5e3, mark=none] {exp(-0.5*ln(x) + ln(0.15) - (-0.5)*ln(2e3)))} node[right,pos=1.0]{$O(N^{-0.5})$};
            \addplot+ [black, domain=2e3:5e3, mark=none] {exp(-0.333*ln(x) + ln(0.4) - (-0.333)*ln(2e3)))} node[right,pos=1.0]{$O(N^{-0.33})$};
            %\addplot+ [black, domain=1e3:15e2, mark=none] {exp(-4/2*ln(x) + ln(1.585) - (-4/2)*ln(1e3)))} node[right,pos=1.0]{$O(N^{-2})$};
            \addlegendentry{Adaptive}
            \addlegendentry{Uniform}
            %\addlegendentry{Adaptive}
        \end{axis}
    \end{tikzpicture}
    \caption{The global error estimator $\eta$ as a function of the number of degrees-of-freedom $N$ in the L-shaped domain case.}
    \label{fig:convergence}
\end{figure}
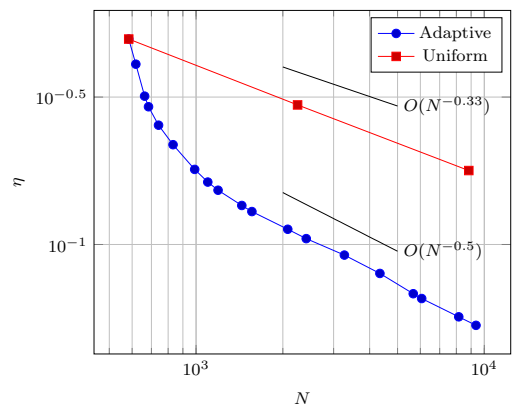

\bibliographystyle{siam}
\bibliography{nitsche_mortaring}

\begin{thebibliography}{10}

\bibitem{AK}
{\sc M.~Ainsworth and D.~W. Kelly}, {\em A posteriori error estimators and
  adaptivity for finite element approximation of the non-homogeneous
  {D}irichlet problem}, Adv. Comput. Math., 15 (2001), pp.~3--23 (2002).

\bibitem{Babuska1973}
{\sc I.~Babu{\v{s}}ka}, {\em {The finite element method with {L}agrangian
  multipliers}}, Numer. Math., 20 (1973), pp.~179--192.

\bibitem{bartels2016numerical}
{\sc S.~Bartels}, {\em Numerical Approximation of Partial Differential
  Equations}, vol.~64 of Texts in Applied Mathematics, Springer International
  Publishing, 2016.

\bibitem{BHS}
{\sc R.~Becker, P.~Hansbo, and R.~Stenberg}, {\em A finite element method for
  domain decomposition with non-matching grids}, ESAIM Math. Model. Numer.
  Anal., 37 (2003), pp.~209--225.

\bibitem{Burman-Hansbo2012}
{\sc E.~Burman and P.~Hansbo}, {\em Fictitious domain finite element methods
  using cut elements: {II}. {A} stabilized {N}itsche method}, Appl. Numer.
  Math., 62 (2012), pp.~328--341.

\bibitem{burman_hansbo_2014}
\leavevmode\vrule height 2pt depth -1.6pt width 23pt, {\em {Fictitious domain
  methods using cut elements: {III}. {A} stabilized {N}itsche method for
  {S}tokes' problem}}, ESAIM Math. Model. Numer. Anal., 48 (2014),
  pp.~859--874.

\bibitem{burman2017penalty}
{\sc E.~Burman, P.~Hansbo, and M.~G. Larson}, {\em {The penalty-free {N}itsche
  method and nonconforming finite elements for the {S}ignorini problem}}, SIAM
  J. Numer. Anal., 55 (2017), pp.~2523--2539.

\bibitem{French-survey}
{\sc F.~Chouly, M.~Fabre, P.~Hild, R.~Mlika, J.~Pousin, and Y.~Renard}, {\em An
  overview of recent results on {N}itsche's method for contact problems}, in
  Geometrically Unfitted Finite Element Methods and Applications, S.~Bordas,
  E.~Burman, M.~Larson, and M.~Olshanskii, eds., vol.~121 of Lecture Notes in
  Computational Science and Engineering, Springer, 2017, pp.~93--141.

\bibitem{chouly2017residual}
{\sc F.~Chouly, M.~Fabre, P.~Hild, J.~Pousin, and Y.~Renard}, {\em
  Residual-based a posteriori error estimation for contact problems
  approximated by {N}itsche's method}, IMA J. Numer. Anal., 38 (2018),
  pp.~921--954.

\bibitem{MR3342214}
{\sc F.~Chouly, P.~Hild, and Y.~Renard}, {\em A {N}itsche finite element method
  for dynamic contact: 1. {S}pace semi-discretization and time-marching
  schemes}, ESAIM Math. Model. Numer. Anal., 49 (2015), pp.~481--502.

\bibitem{MR3342215}
\leavevmode\vrule height 2pt depth -1.6pt width 23pt, {\em A {N}itsche finite
  element method for dynamic contact: 2. {S}tability of the schemes and
  numerical experiments}, ESAIM Math. Model. Numer. Anal., 49 (2015),
  pp.~503--528.

\bibitem{chouly2015symmetric}
\leavevmode\vrule height 2pt depth -1.6pt width 23pt, {\em {Symmetric and
  non-symmetric variants of Nitsche's method for contact problems in
  elasticity: theory and numerical experiments}}, Math. Comp., 84 (2015),
  pp.~1089--1112.

\bibitem{cf}
{\sc W.~Dahmen, B.~Faermann, I.~G. Graham, W.~Hackbusch, and S.~A. Sauter},
  {\em Inverse inequalities on non-quasi-uniform meshes and application to the
  mortar element method}, Math. Comp., 73 (2004), pp.~1107--1138.

\bibitem{MR3633544}
{\sc M.~Fabre, J.~Pousin, and Y.~Renard}, {\em A fictitious domain method for
  frictionless contact problems in elasticity using {N}itsche's method}, SMAI
  J. Comput. Math., 2 (2016), pp.~19--50.

\bibitem{Franca-Stenberg}
{\sc L.~P. Franca and R.~Stenberg}, {\em Error analysis of {G}alerkin least
  squares methods for the elasticity equations}, SIAM J. Numer. Anal., 28
  (1991), pp.~1680--1697.

\bibitem{heinrich2003nitsche}
{\sc B.~Heinrich and S.~Nicaise}, {\em {The Nitsche mortar finite-element
  method for transmission problems with singularities}}, IMA J. Numer. Anal.,
  23 (2003), pp.~331--358.

\bibitem{Juntunen2015}
{\sc M.~Juntunen}, {\em On the connection between the stabilized {L}agrange
  multiplier and {N}itsche's methods}, Numer. Math., 131 (2015), pp.~453--471.

\bibitem{JuntunenEnumath}
{\sc M.~Juntunen}, {\em {On the local mesh size of {N}itsche's method for
  discontinuous material parameters}}, in Numerical Mathematics and Advanced
  Applications ENUMATH 2013, A.~Abdulle, S.~Deparis, D.~Kressner, F.~Nobile,
  and M.~Picasso, eds., vol.~103 of Lecture Notes in Computational Science and
  Engineering, Springer, 2015, pp.~57--63.

\bibitem{JS2012}
{\sc M.~Juntunen and R.~Stenberg}, {\em Nitsche's method for discontinuous
  material parameters}, in Proceedings of the 25th Nordic Seminar on
  Computational Mechanics, K.~Persson, J.~Revstedt, G.~Sandberg, and M.~Wallin,
  eds., Lund University, 2012, pp.~95--98.

\bibitem{LM}
{\sc J.-L. Lions and E.~Magenes}, {\em Non-homogeneous Boundary Value Problems
  and Applications, Vol.~1}, Springer-Verlag Berlin Heidelberg, 1972.

\bibitem{Nitsche}
{\sc J.~Nitsche}, {\em {\"U}ber ein {V}ariationsprinzip zur {L}\"osung von
  {D}irichlet-{P}roblemen bei {V}erwendung von {T}eilr\"aumen, die keinen
  {R}andbedingungen unterworfen sind}, Abh. Math. Sem. Univ. Hamburg, 36
  (1971), pp.~9--15.

\bibitem{Stenberg1995}
{\sc R.~Stenberg}, {\em {On some techniques for approximating boundary
  conditions in the finite element method}}, J. Comput. Appl. Math., 63 (1995),
  pp.~139--148.

\bibitem{Stenberg1998}
\leavevmode\vrule height 2pt depth -1.6pt width 23pt, {\em Mortaring by a
  method of {J}. {A}. {N}itsche}, in Computational Mechanics -- New Trends and
  Applications, S.~Idelsohn, E.~O\~nate, and E.~Dvorkin, eds., CIMNE,
  Barcelona, 1998.

\bibitem{Tartar}
{\sc L.~Tartar}, {\em {An Introduction to {S}obolev Spaces and Interpolation
  Spaces}}, vol.~3 of Lecture Notes of the Unione Matematica Italiana,
  Springer, Berlin, 2007.

\bibitem{Verfurth}
{\sc R.~Verf\"urth}, {\em {A Posteriori Error Estimation Techniques for Finite
  Element Methods}}, Oxford University Press, Oxford, 2013.

\bibitem{wohlmuth2011variationally}
{\sc B.~Wohlmuth}, {\em Variationally consistent discretization schemes and
  numerical algorithms for contact problems}, Acta Numerica, 20 (2011),
  pp.~569--734.

\end{thebibliography}

\end{document}